\documentclass{amsart}
\usepackage{array}
\usepackage{graphicx,verbatim, amsmath, amssymb, amsthm, amsfonts, epsfig, amsxtra,ifthen,mathtools,
caption,enumerate,
hhline,bbm,capt-of,longtable}	
\usepackage[bookmarks=true, bookmarksopen=false,%
    colorlinks=true,%
    linkcolor=darkblue,%
    citecolor=darkblue,%
    filecolor=darkblue,%
    menucolor=darkblue,%
    linktoc=all,%
    urlcolor=darkblue
]{hyperref}

\usepackage[usenames]{color}
\definecolor{darkblue}{cmyk}{1,0.3,0,0.1}  

\newtheorem{proposition}{Proposition}[section]
\newtheorem{theorem}[proposition]{Theorem}
\newtheorem{corollary}[proposition]{Corollary}
\newtheorem{lemma}[proposition]{Lemma}
\newtheorem{prop}[proposition]{Proposition}
\newtheorem{cor}[proposition]{Corollary}
\newtheorem{thm}[proposition]{Theorem}

\theoremstyle{definition}
\newtheorem{example}[proposition]{Example}

\theoremstyle{remark}
\newtheorem{remark}[proposition]{Remark}

\numberwithin{equation}{section}
\usepackage[usenames]{color}

\usepackage{color}

\newcommand{\margincolor}{red}      
\definecolor{darkgreen}{rgb}{0,0.7,0}

\newcounter{margincounter}
\newcommand{\marginnum}{
\ifnum\value{margincounter}<10
\textcolor{\margincolor}{\begin{picture}(0,0)\put(2.2,2.4){\circle{9}}\end{picture}\footnotesize\arabic{margincounter}}
\else\ifnum\value{margincounter}<100
\textcolor{\margincolor}{\begin{picture}(0,0)\put(4.256,2.5){\circle{11}}\end{picture}\footnotesize\arabic{margincounter}}
\else
\textcolor{\margincolor}{\begin{picture}(0,0)\put(6.8,2.5){\circle{14}}\end{picture}\footnotesize\arabic{margincounter}}
\fi\fi
}

\makeatletter
\newcommand{\switchmargin}{
\if@reversemargin
\normalmarginpar
\else
\reversemarginpar
\fi
}
\makeatother

\newcommand{\newword}[1]{\textbf{\emph{#1}}}

\newcommand{\integers}{\mathbb Z}
\newcommand{\rationals}{\mathbb Q}

\newcommand{\reals}{\mathbb R}

\newcommand{\covered}{\lessdot}
\newcommand{\set}[1]{{\left\lbrace #1 \right\rbrace}}
\newcommand{\br}[1]{{\langle #1 \rangle}}

\renewcommand{\P}{{\mathcal P}}
\newcommand{\Q}{{\mathcal Q}}

\newcommand{\T}{{\mathcal T}}
\newcommand{\e}{\mathbf{e}}
\newcommand{\join}{\vee}
\newcommand{\meet}{\wedge}
\renewcommand{\Join}{\bigvee}

\newcommand{\ck}{\spcheck}

\newcommand{\st}{^\mathrm{st}}

\newcommand{\aff}{\mathrm{aff}}
\newcommand{\fin}{\mathrm{fin}}
\newcommand{\inn}{\mathrm{inn}}

\newcommand{\out}{\mathrm{out}}
\newcommand{\tNC}{\widetilde{NC}}
\newcommand{\tNCAc}{\tNC\mathstrut^A_{\!c}}
\newcommand{\tNCAcircc}{\tNC\mathstrut^{A,\circ}_{\!c}}
\newcommand{\tNCCc}{\tNC\mathstrut^C_{\!c}}
\newcommand{\tNCAphic}{\tNC\mathstrut^{A,\phi}_{\!c}}

\newcommand{\curve}{\mathsf{curve}}
\newcommand{\Krew}{\mathsf{Krew}}
\newcommand{\Art}{\operatorname{Art}}

\renewcommand{\mod}[1]{\ (\mathrm{mod}\ #1)}
\newcommand{\cmod}[1]{(\mathrm{mod}\ #1)}  
\newcommand{\SZmodn}{S_\integers\cmod n}
\newcommand{\SZmod}{S_\integers\cmod}
\newcommand{\perm}{\mathsf{perm}}

\newcommand{\afftype}[1]{{\widetilde{\raisebox{0pt}[6pt][0pt]{#1}}}}
\newcommand{\Stilde}{\raisebox{0pt}[0pt][0pt]{$\,{\widetilde{\raisebox{0pt}[6.1pt][0pt]{\!$S$}}}$}}
\newcommand{\Stildes}[1][2n]{\Stilde^{\mathrm{s}}_{#1}}

\title{Noncrossing partitions of an annulus}
\author{Laura G. Brestensky}
\author{Nathan Reading}
\thanks{Nathan Reading was partially supported by the Simons Foundation under award number 581608 and by the National Science Foundation under award number DMS-2054489.  Laura Brestensky was partially supported by the National Science Foundation under award number DMS-2054489.}
\subjclass[2010]{Primary: 20F55, 05E16, 20F36}

\begin{document}

{\renewcommand{\afftype}[1]{{\widetilde{\raisebox{0pt}[5pt][0pt]{#1}}}}
\begin{abstract}
The noncrossing partition poset associated to a Coxeter group $W$ and Coxeter element $c$ is the interval $[1,c]_T$ in the absolute order on $W$.
We construct a new model of noncrossing partititions for $W$ of classical affine type, using planar diagrams (affine types $\afftype{A}$ and $\afftype{C}$ in this paper and affine types $\afftype{D}$ and $\afftype{B}$ in the sequel).
The model in type $\afftype{A}$ consists of noncrossing partitions of an annulus.
In type~$\afftype{C}$, the model consists of symmetric noncrossing partitions of an annulus or noncrossing partitions of a disk with two orbifold points.
Following the lead of McCammond and Sulway, we complete $[1,c]_T$ to a lattice by factoring the translations in $[1,c]_T$, but the combinatorics of the planar diagrams leads us to make different choices about how to factor.
\end{abstract}
\maketitle}


\setcounter{tocdepth}{2}
\tableofcontents


\section{Introduction}\label{introduction}
Noncrossing partitions associated to a Coxeter group are algebraic/combinatorial objects that figure prominently in the analysis of the associated Artin group.
In the classical finite types A, B, and D (and to some extent in other types---see~\cite{plane}), noncrossing partitions are best understood in terms of certain planar diagrams.
This paper and its sequel~\cite{affncD} extend planar diagrams for noncrossing partitions to the classical affine types $\afftype{A}$ and $\afftype{C}$ (in this paper) and $\afftype{D}$ and $\afftype{B}$ (in the sequel).
Furthermore, in~\cite{surfnc}, which can be thought of as a ``prequel'' to these papers, the combinatorics (but not the algebra) of these planar diagrams is generalized to the setting of marked surfaces (in the sense of the marked surfaces model for cluster algebras).
More specifically, planar diagrams for noncrossing partitions of types A and $\afftype{A}$ are generalized by noncrossing partitions of marked surfaces (without ``punctures'') and planar diagrams for noncrossing partitions of the other classical finite and affine types are generalized by symmetric noncrossing partitions of a marked surface with (or without) double points.

Noncrossing partitions of a cycle were introduced by Kreweras~\cite{Kreweras}.
Biane~\cite{Biane} connected noncrossing partitions of a cycle to finite Coxeter groups of type A (the symmetric groups) and showed that the lattice of noncrossing partitions is isomorphic to $[1,c]_T$, the interval between the identity and a Coxeter element $c$ in the absolute order.
The analogous interval in a finite Coxeter group of type B is modeled by centrally symmetric noncrossing partitions \cite{Ath-Rei,Reiner}, and there is an analogous planar model~\cite{Ath-Rei} for type~D.
These planar models can be understood uniformly in terms of the Coxeter plane, which was defined and studied in~\cite{RegPoly,Steinberg}.
Specifically, realizing group elements as permutations of an orbit, one projects to the Coxeter plane and interprets the cycle structure of the permutations as blocks in a partition~\cite{plane}.

The interval $[1,c]_T$ in a finite Coxeter group $W$ serves as a Garside structure for the corresponding spherical Artin group $\Art(W)$, leading to a ``dual presentation'' of $\Art(W)$~\cite{Bessis,Bra-Wa} and proving desirable properties of $\Art(W)$.
The fact that $[1,c]_T$ is a lattice is crucial to the dual presentation.

Outside of finite type, the interval $[1,c]_T$ need not be a lattice.  
The case of affine type is treated in a series of papers that begins by extending crucial results on rigid motions~\cite{BWOrth} to affine type~\cite{Br-McC} and continues with an analysis of the failure of the lattice property in Coxeter groups of affine type~\cite{McFailure}.
The series culminates in \cite{McSul}, in which McCammond and Sulway extend the affine Coxeter group $W$ to a larger group, thereby extending the interval $[1,c]_T$ to a lattice.
The larger lattice serves as a Garside structure for a supergroup of the Euclidean Artin group $\Art(W)$, which inherits desirable (previously conjectured) properties from the supergroup.

The purpose of this paper is to provide planar models, in the classical affine types, for the intervals $[1,c]_T$ and the larger lattices constructed by McCammond and Sulway.
We follow the idea of \cite{plane} and obtain models by projecting an orbit to the analog of the Coxeter plane.  
In type $\afftype{A}$, the resulting model consists of \newword{noncrossing partitions of an annulus}, defined in Section~\ref{nc ann sec} and exemplified in Figure~\ref{nc exs}.
\begin{figure}
\begin{tabular}{cccc}
\includegraphics{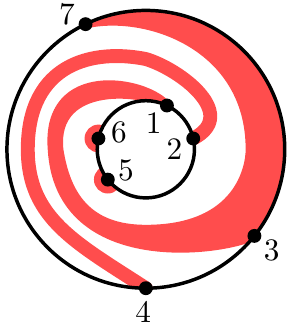}
&\includegraphics{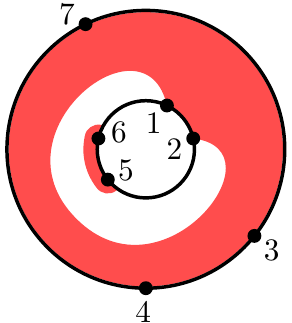}
&\includegraphics{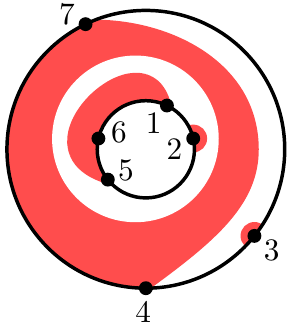}
\end{tabular}
\caption{Some noncrossing partitions of an annulus}
\label{nc exs}
\end{figure}

In type $\afftype{C}$, the model consists of \newword{symmetric noncrossing partitions of an annulus}, defined in Section~\ref{sym ann sec}.
Some examples are shown in the top row of Figure~\ref{nc C exs}.
\begin{figure}
\begin{tabular}{ccc}
\scalebox{0.9}{\includegraphics{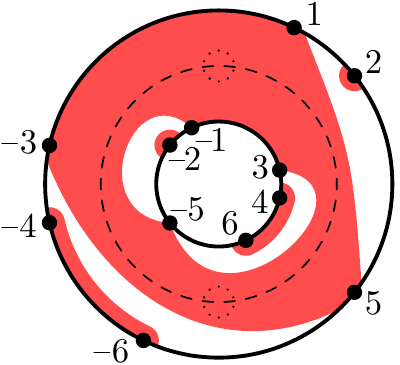}}
&\scalebox{0.9}{\includegraphics{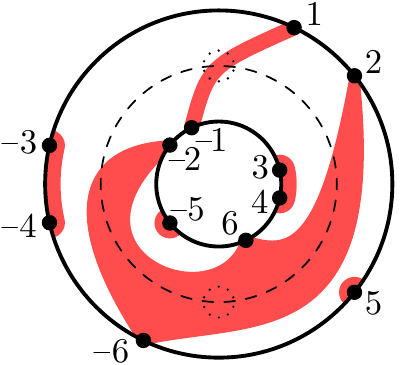}}
&\scalebox{0.9}{\includegraphics{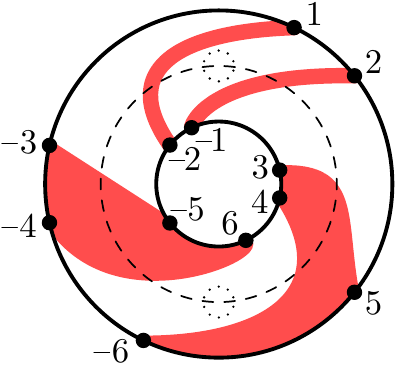}}\\
\scalebox{0.9}{\includegraphics{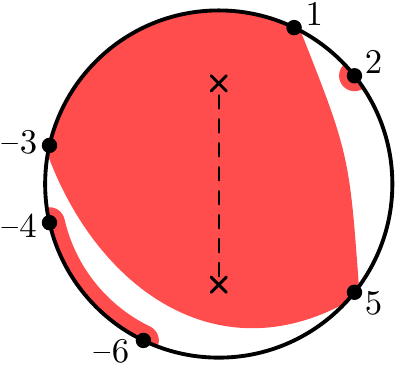}}
&\scalebox{0.9}{\includegraphics{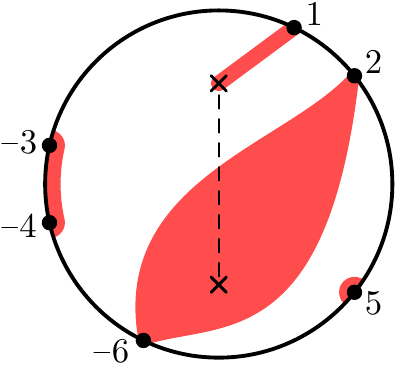}}
&\scalebox{0.9}{\includegraphics{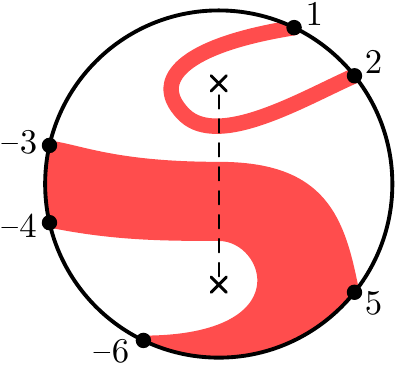}}
\end{tabular}
\caption{Some symmetric noncrossing partitions of an annulus and corresponding noncrossing partitions of the two-orbifold disk}
\label{nc C exs}
\end{figure}
To see the symmetry, think of the pictures as perspective drawings of a circular cylinder;
the symmetry is a rotation of the cylinder along a vertical axis.
The intersections of this axis with the cylinder are indicated in the pictures by small dotted circles.
(Section~\ref{sym ann sec} contains a formula for the symmetry in polar coordinates.)
Passing to the quotient modulo the symmetry, symmetric noncrossing partitions of an annulus become \newword{noncrossing partitions of a two-orbifold disk} (Section~\ref{nc 2 orb sec}), as shown in the bottom row of Figure~\ref{nc C exs}.
The fixed points of the symmetry in the annulus become order-$2$ orbifold points in the disk, marked with an $\times$ in the pictures.  
The dotted circle cuts the annulus into two pieces (a fundamental domain of the symmetry and its one translate).
The image of the circle under the quotient map is shown as a dotted line.

The key idea behind McCammond and Sulway's extension of $[1,c]_T$ to a lattice is the notion of ``factoring'' translations in an affine Coxeter group $W$.
Certain elements of $W$ act (in the affine representation) as translations, and the translations in $[1,c]_T$ were implicated (in~\cite{McFailure}) in the failure of the lattice property.
To overcome the failure, each translation in $[1,c]_T$ is factored into two or more translations not already contained in $W$.
The larger group containing $W$ is generated by $W$ and these new \newword{factored translations}.

One of the motivating questions for the current project was how factored translations would fit into planar models.
Would projecting the Coxeter plane lead to a useful model of $[1,c]_T$ only, or of the larger interval (incorporating factored translations) only, or of both?
In type $\afftype{A}$, the best possible thing happened:  The planar diagrams model the larger interval, and the elements of the smaller interval are distinguished by a simple criterion.
The elements of $[1,c]_T$ are the noncrossing partitions of the annulus with no \newword{dangling annular blocks}.
A dangling annular block is a block that is an annulus and has numbered points on only one component of its boundary.
Thus the leftmost and middle picture in Figure~\ref{nc exs} represent elements of $[1,c]_T$, while the rightmost picture represents an element in the larger interval, not in $[1,c]_T$.

\begin{remark}\label{thesis}
Most of the results reported here have appeared in Laura Brestensky's thesis~\cite{BThesis}.
However, for efficiency of exposition, some of the results that were proven directly in~\cite{BThesis} are here cited as special cases of results that now appear in the prequel~\cite{surfnc}.
Other parts of this paper are revisions of material from~\cite{BThesis}.
\end{remark}

\begin{remark}[Relationship to other work]
Several other papers consider constructions that are superficially or more deeply similar to the constructions in this paper.
\begin{itemize}
\item
McCammond and Sulway's work and the results of this paper were foreshadowed by \cite{Digne2}, in which Digne proved the lattice property for a certain choice of Coxeter element $c$ of an affine Coxeter group of type $\afftype{A}$.
The proof involved representing elements of $[1,c]_T$ as certain noncrossing paths in an annulus.
\item
The combinatorial model also recalls even earlier work of Allcock~\cite{Allcock}.
Each reflection or factored translation is associated to a block that connects two numbered points, connects a numbered point to itself, or connects a numbered point to an orbifold point.  
Such a block can be identified with a simple braid so that the interval group (see Section~\ref{McSul general sec}) realizes elements of the Artin group as elements of a braid group (of the plane minus~$1$ or $2$ points for~$\afftype A$ or~$\afftype C$), as in~\cite{Allcock}.
\item
The ``annular noncrossing partitions'' of Mingo and Nica~\cite{MingoNica} have some similarity to our noncrossing partitions of an annulus, but the construction in~\cite{MingoNica} is explicitly tied to finite symmetric groups, and thus does not allow annular blocks and has only finitely many noncrossing partitions for a fixed number of numbered points.
\item
There are some similarities between noncrossing partitions of an annulus and the chord diagrams and arc diagrams that appear in \cite{Maresca} in connection with representation-theory.
\end{itemize}
\end{remark}

\subsection*{Outline of the paper}
After some general background in Section~\ref{back sec}, we consider affine type $\afftype{A}$ in Section~\ref{aff type a}.
We construct the usual representation of the \mbox{type-$\afftype{A}$} Coxeter group as the group $\Stilde_n$ of affine permutations and consider the larger group $\SZmodn$ of periodic permutations (permutations with mod-$n$ symmetry, without the additional condition that defines affine permutations).
We construct the Coxeter plane and show that the projection of an orbit to the Coxeter plane yields an infinite strip with translational symmetry that leads naturally to a model on the annulus.
We define noncrossing partitions of this annulus and quote results of~\cite{surfnc} that show that these noncrossing partitions form a graded lattice and give a simple description of its cover relations and rank function.
We also describe Kreweras complementation in the lattice.
We show that the poset of noncrossing partitions with no dangling annular blocks is isomorphic to the interval $[1,c]_T$ and that the full lattice of noncrossing partitions is isomorphic to the analogous interval in the larger group $\SZmodn$.
The essence of the isomorphism is to read blocks of a noncrossing partition as cycles in a permutation.

In Section~\ref{aff type c}, we consider affine type $\afftype{C}$, using the standard notion of ``folding'' to reuse the work done in type $\afftype{A}$ and obtain the analogous results in type $\afftype{C}$.

In Section~\ref{McSul sec}, we connect our constructions to the work of McCammond and Sulway.
The combinatorial model of noncrossing partitions of an annulus suggests its own way of factoring translations in type $\afftype{A}$:
A translation in $[1,c]_T$ corresponds to a noncrossing partitions whose only nontrivial block is an annulus containing one numbered point on each component of its boundary.
A natural way to factor such an element is by breaking the annular block into two dangling annular blocks. 
Surprisingly, this factorization is \emph{not} the same as the factorization given in~\cite{McSul}.  
Instead, there is a continuous family of schemes for factorization, all giving isomorphic intervals and isomorphic supergroups of the Artin group.
Out of this continuous family, the combinatorics picks out what is perhaps the tidiest factorization scheme, where the generating set of the larger group is closed under conjugation and the larger group is $\SZmodn$, independent of the choice of Coxeter element~$c$.
In type~$\afftype{C}$, there is no need to factor translations.
We will see in~\cite{affncD} that in type $\afftype{D}$, the combinatorics again chooses a tidy scheme for factorization, but the relationship to the McCammond-Sulway factorization scheme is significantly more complicated.

\subsection*{Acknowledgments}
Thanks to Jon McCammond for suggesting how to find the Coxeter plane in $V^*$. 
Thanks also to Monica Vazirani for helpful information about representation-theoretic connections to the extended affine symmetric group.
Thanks to the anonymous referees for helpful suggestions.
An extended abstract of this paper~\cite{BR-FPSAC} was published in the proceedings of the FPSAC 2023 conference.

\section{Background}
\label{back sec}

In this section, we review some background on Coxeter groups and root systems.
We assume the most basic definitions and facts about Coxeter groups.

\subsection{Coxeter groups and root systems}\label{cox sec}

Let $(W,S)$ be a Coxeter system.
We recall the construction of a \newword{Cartan matrix} associated to $W$.
Write $S=\set{s_1,\ldots,s_n}$ and write $m(s_i,s_j)$ for the order of $s_is_j$ in $W$.  
A Cartan matrix $A$ is an $n\times n$ matrix $[a_{ij}]$ with $a_{ii}=2$ for $i=1,\ldots,n$, with non-positive off-diagonal entries such that $a_{ij}a_{ji}=4\cos\frac\pi{m(i,j)}$.
The Cartan matrix $A$ is assumed to be \newword{symmetrizable}, meaning that there exist positive real numbers $d_1,\ldots,d_n$ such that $d_i a_{ij}=d_j a_{ji}$ for all $i,j$.
In the cases of interest here, $A$ can be taken to have integer entries.

The choice of a symmetrizable Cartan matrix $A$ for $W$ specifies a reflection representation of $W$.
Let $V$ be a real vector space with basis $\alpha_1,\ldots,\alpha_n$.
For each~$i$, we define $\alpha_i\ck$ to be $d_i^{-1} \alpha_i$, where the $d_i$ are symmetrizing constants, as above.
We define a symmetric bilinear form $K$ on $V$ by $K(\alpha_i\ck,\alpha_j)=a_{ij}$.
We have $K(d_i^{-1}\alpha_i,\alpha_i)=2$, so ${d_i=\frac12K(\alpha_i,\alpha_i)}$.
In the reflection representation of~$W$ on $V$, the simple reflection $s_i$ acts on a vector $x$ by $s_i(x)=x-K(\alpha\ck_i,x)\alpha_i$.
In particular, $s_i(\alpha_j)=\alpha_j-a_{ij}\alpha_i$ and $s_i(\alpha\ck_j)=\alpha\ck_j-a_{ji}\alpha\ck_i$ for each $j$.

The set of \newword{real roots} associated to $A$ is $\set{w\alpha_i:w\in W,i=1,\ldots,n}$.
Each root $\beta=w\alpha_i$ has an associated co-root $\beta\ck=w\alpha_i\ck$.
The \newword{positive roots} are the roots in the nonnegative span of the simple roots.

The \newword{reflections} in $W$ are the elements $\set{wsw^{-1}:w\in W,s\in S}$.
These are precisely the elements of $W$ that act as reflections (i.e.\ have an $(n-1)$-dimensional fixed space and a $(-1)$-eigenvector) in the reflection representation of $W$.
Specifically, given a reflection $t=ws_iw^{-1}$, the root $\beta=w\alpha_i$ is its $(-1)$-eigenvector and the hyperplane $\set{v\in V:K(v,\beta)=0}$ is its fixed space.
The reflection $t$ sends $x\in V$ to $x-K(\beta\ck_i,x)\beta$.

When $W$ is finite, the real roots constitute the \newword{root system} $\Phi$ associated to~$A$.
When $W$ is infinite, the root system includes some additional vectors called \newword{imaginary roots}.
We will not explicitly deal with imaginary roots, but the vector $\delta$ appearing in Section~\ref{aff sec} is the positive imaginary root closest to zero.

The reflection representation of $W$ on $V$ has a dual representation as linear transformations of the dual space $V^*$, constructed in the usual way.
Let $\rho_1,\ldots,\rho_n$ be the basis of $V^*$ that is dual to the simple \textit{co-roots} basis $\alpha_1\ck,\ldots,\alpha_n\ck$.
The vectors $\rho_1,\ldots,\rho_n$ are the \newword{fundamental weights} associated to $A$.  
The reflection $s_i$ acts on $V^*$ by fixing $\rho_j$ for $j\neq i$ and sending $\rho_i$ to $\rho_i-\sum_{k=1}^na_{ki}\rho_k$.
More generally the reflection $t$ acts on $V^*$ by fixing the hyperplane $\set{x\in V^*:\br{x,\beta}=0}$ and negating the vector $K(\,\cdot\,,\beta)\in V^*$.

\begin{remark}\label{iconoclast}
We break with the Lie-theoretic tradition of putting roots and co-roots in spaces dual to each other and putting roots and weights in the same space.
The way of constructing roots and weights in this paper is essential to the construction introduced in Section~\ref{aff sec} and carried out in Section~\ref{plane A sec}, where we project an orbit in $V$ to the Coxeter plane in $V^*$.
\end{remark}

\subsection{Coxeter elements and noncrossing partitions}\label{background ncp}
A \newword{Coxeter element} is an element $c\in W$ that can be written as a product, in some chosen order, of the elements of $S$, each repeated exactly once in the product.
Coxeter elements of $W$ correspond to acyclic orientations of the Coxeter diagram:
An edge in the Coxeter diagram is oriented $s_i\to s_j$ if and only if $s_i$ precedes $s_j$ in every reduced word for~$c$.
A simple reflection $s_i$ is a source in the acyclic orientation associated to $c$ if and only if there is a reduced word for $c$ starting with~$s_i$.
Similarly, $s_i$ is a sink if and only if it can occur as the last letter of a reduced word for $c$.
Performing a \newword{source-sink move} on an acyclic orientation means choosing a source or sink and reversing all arrows on it.
If $s_i$ is a source or sink, then the source-sink move at $s_i$ corresponds to conjugating $c$ by $s_i$.
When the Coxeter diagram is a tree (and thus in finite type and all affine types except $\afftype{A}_{n-1}$), all Coxeter elements are conjugate by source-sink moves.
In type $\afftype{A}_{n-1}$, every acyclic orientation can be transformed by source-sink moves to an orientation where $s_n$ is the unique source and there is a unique sink~$s_k$.

The Coxeter element $c$ determines a skew-symmetric bilinear form $\omega_c$ on $V$ by
\begin{equation}\label{omega def}
\omega_c(\alpha_i\ck,\alpha_j)=
\left\lbrace\begin{array}{ll}
a_{ij}&\text{if }s_i\text{ follows }s_j\text{ in }c,\\
0&\text{if }i=j,\text{ or}\\
-a_{ij}&\text{if }s_i\text{ precedes }s_j\text{ in }c.
\end{array}\right.
\end{equation}
As a consequence of \cite[Lemma~3.8]{typefree}, this form satisfies $\omega_c(cx,cy)=\omega_c(x,y)$.

Since $S\subseteq T$, the set $T$ of reflections generates $W$.
The \newword{absolute length} or \newword{reflection length} $\ell_T(w)$ of an element $w\in W$ is the number of letters in a shortest expression for $w$ as a product of elements of $T$.
Such a shortest expression is called a \newword{reduced $T$-word} for $w$.
The \newword{absolute order} $\le_T$ is the partial order on $W$ defined by $u\le_Tw$ if and only if $\ell_T(w)=\ell_T(u)+\ell_T(u^{-1}w)$.
Thus $u\le_Tw$ if and only if there is a reduced $T$-word for $w$ that has a reduced $T$-word for~$u$ as a prefix.

The starting point for this paper is the interval $[1,c]_T$ in the absolute order between the identity and a Coxeter element $c$.
When $W$ is finite, this is a lattice~\cite{BWlattice}, often called the \newword{$W$-noncrossing partition lattice}, but when $W$ is infinite, $[1,c]_T$ can fail to be a lattice.  
The $W$-noncrossing partition lattice is named after the finite type-A case, where it is modeled by noncrossing partitions of a cycle~\cite{Biane,Kreweras}.
There are also planar models in types B and~D~\cite{Ath-Rei,Reiner}.

When $W$ is of finite type, there is a special plane, called the \newword{Coxeter plane}, on which $c$ acts as a rotation by $\frac{2\pi}h$.
Here, $h$ is the \newword{Coxeter number} (the order of $c$).
The Coxeter plane figures prominently in \cite{RegPoly}, and a uniform construction of the plane is found in~\cite{Steinberg}.

The Coxeter plane is closely related to the planar diagrams for noncrossing partitions in types A, B, and D.
The relationship, explored in \cite{plane}, is as follows:
Given a finite Coxeter group $W$, find a smallest orbit $o$. 
Any element $w$ of $W$ acts as a permutation of $o$, and in particular decomposes $o$ into cycles.
One projects $o$ orthogonally onto the Coxeter plane and considers the set partition defined on the projected orbit by this cycle decomposition.
In types A, B, and D (and also $H_3$ and $I_2(m)$), there are simple criteria to decide, from the resulting planar diagram, whether $w\in[1,c]_T$.
This construction recovers the diagrams in \cite{Ath-Rei,Biane,Kreweras,Reiner}.

In this paper and its sequel~\cite{affncD}, we implement the ``project a small orbit to the Coxeter plane'' construction to create planar diagrams for $[1,c]_T$ in the classical affine types.

\subsection{Folding the poset of noncrossing partitions}\label{folding sec}
In this section, we discuss how the poset $[1,c]_T$ of noncrossing partitions behaves under folding.
We do not discuss the general problem of folding, but rather take easy sufficient hypotheses on two Coxeter groups to relate the corresponding posets of noncrossing partitions.

Suppose $(W',S')$ is a Coxeter system.
A \newword{diagram automorphism} of $(W',S')$ is a permutation $\phi$ of $S'$ such that $m(\phi(r),\phi(s))=m(r,s)$ for all $r,s\in S'$.
A diagram automorphism $\phi$ extends uniquely to a group automorphism of $W'$, and we re-use the name $\phi$ for the group automorphism.
Suppose further that elements in each $\phi$-orbit in $S'$ commute pairwise, so that in particular the product of the elements of the orbit has order $2$.
Let $S$ be the set of such products.
Let $W$ be the subgroup of $W'$ consisting of elements fixed by $\phi$.
If $(W,S)$ is a Coxeter system, then we say that $(W,S)$ is a \newword{folding} of $(W',S')$.

Because the elements of $S$ are products over the $\phi$-orbits of $S'$, every Coxeter element $c$ of $W$ is also a Coxeter element of $W'$.
A Coxeter element of $W'$ is a Coxeter element of $W$ if and only if the corresponding orientation of the Coxeter diagram of $W'$ is preserved by the diagram automorphism~$\phi$.

For the rest of the section, we assume that $(W,S)$ is a folding of $(W',S')$ and that $c$ is a Coxeter element of $W$ (and thus also of $W'$).
Write $T'$ for the set of reflections in $W'$ and $T$ for the set of reflections in $W$.
We will prove the following proposition.

\begin{proposition}\label{fold subposet}
The interval $[1,c]_T$ in $W$ is the subposet of the interval $[1,c]_{T'}$ in $W'$ induced by $W\cap[1,c]_{T'}$.
In other words, $[1,c]_T$ is the subposet of $[1,c]_{T'}$ induced by the set of elements fixed by $\phi$.
\end{proposition}

It is easy and well known that for any lattice $L$ and any automorphism $\phi$ of~$L$, the subposet $L^\phi$ induced by elements of $L$ fixed by $L$ is a sublattice of $L$.
(If $x,y\in L^\phi$, then $\phi(x\join y)=\phi(x)\join\phi(y)=x\join y$, and similarly for meets.)
Thus we have the following corollary of Proposition~\ref{fold subposet}.

\begin{corollary}\label{fold sublattice}
If $[1,c]_{T'}$ is a lattice, then $[1,c]_T$ is a sublattice of $[1,c]_{T'}$.
\end{corollary}

The proof of Proposition~\ref{fold subposet} will make use of several lemmas.

\begin{lemma}\label{phi auto general}
The map $\phi$ is an automorphism of the absolute order on $W'$.
\end{lemma}
\begin{proof}
Each $t\in T'$ is $wsw^{-1}$ for some $w\in W'$ and $s\in S'$ and since $\phi$ is a group automorphism that permutes $S'$, $\phi(t)=\phi(w)\phi(s)[\phi(w)]^{-1}$ is also a reflection.
Applying the same argument for $\phi^{-1}$, we see that $\phi$ permutes the set $T'$ of reflections in $W'$.
Thus $\phi$ also preserves the absolute length function $\ell_{T'}$ in $W'$.
Thus $u\le_{T'}w$ if and only if $\ell_{T'}(u)+\ell_{T'}(u^{-1}w)=\ell_{T'}(w)$, if and only if $\ell_{T'}(\phi(u))+\ell_{T'}(\phi(u^{-1}w))=\ell_{T'}(\phi(w))$ if and only if (because $\phi$ is a group automorphism) $\phi(u)\le_{T'}\phi(w)$.
\end{proof}

\begin{lemma}\label{ref orb}
Each reflection $t\in T$ is the product of a $\phi$-orbit in $T'$.
\end{lemma}
\begin{proof}
Each $s\in S$ is the product of a $\phi$-orbit in $S'$ and each $w\in W$ is fixed by~$\phi$.
If $t=wsw^{-1}$, then write $s=a_1\cdots a_k$ with $a_{i+1}=\phi(a_i)$ for $i=1,\ldots,k$ with indices mod $k$.
Then $\phi(wa_iw^{-1})=wa_{i+1}w^{-1}$ and $t=(wa_1w^{-1})\cdots(wa_kw^{-1})$.
\end{proof}

The \newword{multiplicity} of a reflection $t\in T$ is the size of the $\phi$-orbit in $T'$ whose product is $t$.

\begin{lemma}\label{multiplicities}
Let $c$ be a Coxeter element of $W$.
If $a_1\cdots a_n$ is a reduced $T$-word for $c$, then the multiset of multiplicities of the letters $a_1,\ldots,a_n$ is the same as the multiset of $\phi$-orbit sizes in $S'$.
\end{lemma}
\begin{proof}
There is at least one reduced $T$-word for $c$ that satisfies the lemma by definition, namely, any reduced $T$-word for $c$ whose letters are $S$.
Hurwitz moves on reduced $T$-words are moves that take two adjacent letters $tt'$ and replace them by either $t'(t'tt')$ or $(tt't)t$.
It is apparent from the proof of Lemma~\ref{ref orb} that Hurwitz moves preserve multiplicities of the letters of the $T$-word.
Since any two reduced $T$-words are related by a sequence of Hurwitz moves~\cite[Theorem~1.4]{IgusaSchiffler}, the lemma follows.
\end{proof}

\begin{lemma}\label{fold red}
Suppose $w\in[1,c]_T$ and suppose $a_1\cdots a_k$ is a reduced $T$-word for~$w$.
Let $b_1\cdots b_l$ be a word obtained from $a_1\cdots a_k$ by replacing each $a_i\in T$ by a formal product of the corresponding $\phi$-orbit in $T'$.
Then $b_1\cdots b_l$ is a reduced $T'$\nobreakdash-word for~$w$.
\end{lemma}
\begin{proof}
Since $w\in[1,c]_T$, there exist letters $a_{k+1},\ldots,a_n$ in $T$ such that $a_1\cdots a_n$ is a reduced $T$-word for $c$.
Let $b_1\cdots b_m$ be the $T'$-word obtained from $a_1\cdots a_n$ as described in the lemma.
Lemma~\ref{multiplicities} implies that $m$ is the sum of the sizes of the $\phi$-orbits in $S'$ corresponding to the elements of $S$.
That is, $m=|S'|$.
We see that $b_1\cdots b_m$ is a reduced $T'$-word for $c$.
Since $b_1\cdots b_l$ is a prefix of $b_1\cdots b_m$, it is a reduced $T'$-word for $w$, as desired.
\end{proof}

\begin{proof}[Proof of Proposition~\ref{fold subposet}]
The proposition amounts to two assertions:
First, that an element $u\in W$ has $u\le_Tc$ if and only if $u\le_{T'}c$; and second that two elements $u,w\in[1,c]_T$ have $u\le_Tw$ if and only if $u\le_{T'}w$.

Suppose $u\le_Tc$ and let $a_1\cdots a_n$ be a reduced $T$-word for $c$ having a prefix $a_1\cdots a_k$ that is a $T$-word for $u$.
By Lemma~\ref{fold red}, there is a reduced $T'$-word for $c$ having a prefix that is a $T'$-word for $u$, so $u\le_{T'}c$.
Conversely, suppose $u\in W$ has $u\le_{T'}c$ and let $b_1\cdots b_m$ be a reduced $T'$-word for $c$ having a prefix $b_1\cdots b_l$ that is a reduced $T'$-word for $u$.
By Lemma~\ref{fold red}, we can take $b_1\cdots b_l$ to be the word obtained from a reduced $T$-word $a_1\cdots a_k$ for $u$.
Since also $u^{-1}c$ is in $W$, we can also appeal to Lemma~\ref{fold red} to take $b_{l+1}\cdots b_m$ to be the word obtained from a reduced $T$-word $a_{k+1}\cdots a_{n'}$ for $u$.
But then, the $T'$-word $b_1\cdots b_m$ for $c$ is obtained from $a_1\cdots a_{n'}$ as described in Lemma~\ref{fold red}, so Lemma~\ref{multiplicities} implies that $n'=n$ and thus $a_1\cdots a_n$ is a reduced $T$-word for $c$.
We see that $u\le_Tc$.

Now suppose $u,w\in[1,c]_T$.
If $u\le_Tw$ then there is a reduced $T$-word for $w$ having a reduced $T$-word for $u$ as a prefix.
Lemma~\ref{fold red} lets us construct a reduced $T'$-word for $w$ having a reduced $T'$-word for $u$ as a prefix, so $u\le_{T'}w$.
Conversely, if $u\le_{T'}w$, then there exists a reduced $T'$-word $b_1\cdots b_m$ for $c$ having a prefix $b_1\cdots b_k$ that is a word for $u$ and a prefix $b_1\cdots b_l$ that is a word for $w$, with $k\le l$.
Since $u$, $u^{-1}w$ and $w^{-1}c$ are all in $W$, arguing as in the previous paragraph, we can take $b_1\cdots b_m$ to be the word obtained from a reduced $T$-word $a_1\cdots a_n$ for $c$ having a prefix that is a $T$-word for $w$ and a shorter prefix that is a $T$-word for $u$.
Thus $u\le_Tw$.
\end{proof}

\subsection{Affine type}\label{aff sec}
The Coxeter groups of \newword{affine type} are precisely the Coxeter groups that admit a Cartan matrix $A$ with determinant $0$ such that every principal minor of $A$ has positive determinant.
Equivalently, $K$ is positive semi-definite but not positive definite, and for each $i=1,\ldots,n$, the restriction of $K$ to the span of ${\set{\alpha_1,\ldots,\alpha_n}\setminus\set{\alpha_i}}$ is positive definite.
There is a well known classification of Coxeter groups of affine type.
Some Coxeter groups of affine type admit multiple integer symmetrizable Cartan matrices, but we will make the standard choice (which comes from a uniform construction starting with a finite crystallographic Coxeter group.  
See, for example, \cite[Chapter~4]{Humphreys}.)
We will call this choice a \newword{standard affine Cartan matrix} associated to $W$.

We will not need the details of the construction of a standard affine Cartan matrix~$A$, but we will describe some of the properties of such an $A$.
There is a vector~$\delta\in V$ whose simple-root coordinates constitute a $0$-eigenvector of $A$ with strictly positive entries such that the gcd of the entries is~$1$.
There is an index ${\aff\in\set{1,\ldots,n}}$ that makes the rest of this paragraph true.
The matrix $A_\fin$ obtained from $A$ by deleting row $\aff$ and column $\aff$ is the Cartan matrix associated to an irreducible Coxeter group $W$ of finite type.
Let $V_\fin$ be the span of $\set{\alpha_1,\ldots,\alpha_n}\setminus\set{\alpha_\aff}$, let~$\Phi_\fin$ be the finite root system $\Phi\cap V_\fin$, and let $W_\fin$ be the associated Coxeter group (the standard parabolic subgroup of $W$ generated by $S\setminus\set{s_\aff}$).
The real roots associated to $A$ are precisely the vectors $\beta+k\delta$ for $\beta\in\Phi_\fin$ and $k\in\integers$.
(The imaginary roots associated to $A$ are the nonzero integer multiples of $\delta$.)


As mentioned in Section~\ref{background ncp}, in finite type, for each Coxeter element $c$ there is a plane called the Coxeter plane on which $c$ acts as a rotation.
In the affine case, there are two analogs of the Coxeter plane, one contained in $V$ and the other in~$V^*$.

In an affine Coxeter group, the action of a Coxeter element $c$ on $V$ has the eigenvalue $1$ with multiplicity $2$, but has only a $1$-dimensional fixed space.
The fixed space is spanned by $\delta$ (which indeed is fixed by every element of $W$).
There is a unique generalized $1$-eigenvector $\gamma_c$ associated to $\delta$ contained in the subspace~$V_\fin$.
(The fact that $\gamma_c$ is a generalized $1$-eigenvector associated to $\delta$ means that ${(c-1)\gamma_c=\delta}$.
See \cite[Proposition~3.1]{afforb}.)
The plane spanned by $\delta$ and $\gamma_c$ is fixed as a set by $c$, and will be called the \newword{Coxeter plane in $V$}.  

The action of the Coxeter element $c$ on $V^*$ also has the eigenvalue $1$ with multiplicity $2$ and a $1$-dimensional fixed space.
The vector $\omega_c(\delta,\,\cdot\,)\in V^*$ spans the fixed space \cite[Lemma~3.5]{afforb}.
\begin{lemma}\label{gamma*}
The vector $\omega_c(\gamma_c,\,\cdot\,)\in V^*$ is a generalized $1$-eigenvector for $c$, associated to the $1$-eigenvector $\omega_c(\delta,\,\cdot\,)$.
\end{lemma}
\begin{proof}
We need to show that $c\cdot\omega_c(\gamma_c,\,\cdot\,)=\omega_c(\gamma_c,\,\cdot\,)+\omega_c(\delta,\,\cdot\,)$.
Since $\gamma_c$ is a generalized $1$-eigenvector associated to $\delta$, we know that $c\gamma_c=\gamma_c+\delta$.
Thus $\omega_c(c\gamma_c,\,\cdot\,)=\omega_c(\gamma_c,\,\cdot\,)+\omega_c(\delta,\,\cdot\,)$.
We will verify that $\omega_c(c\gamma_c,\,\cdot\,)=c\cdot\omega_c(\gamma_c,\,\cdot\,)$.
Writing $x$ for $\omega_c(\gamma_c,\,\cdot\,)$, for any $y\in V$, we have 
\[\br{cx,y}=\br{x,c^{-1}y}=\omega_c(\gamma_c,c^{-1}y)=\omega_c(c\gamma_c,y),\]
as desired.
\end{proof}

The plane in $V^*$ spanned by $\omega_c(\delta,\,\cdot\,)$ and $\omega_c(\gamma_c,\,\cdot\,)$ will be called the \newword{Coxeter plane in $V^*$}.

To implement the ``project a small orbit to the Coxeter plane'' construction in affine type, we need to decide whether to consider an orbit in $V$ or $V^*$ and whether to project to the Coxeter plane in $V$ or $V^*$.
We will see here and in~\cite{affncD} that taking an orbit in $V$ and projecting to the Coxeter plane in $V^*$ works well.
Other choices appear to work less well.

\section{Affine type A}\label{aff type a}
\subsection{Affine permutations and a larger group}
The Coxeter group $W$ of type $\afftype{A}_{n-1}$ has $S=\set{s_1,\ldots,s_n}$, $m(s_i,s_{i+1})=3$, and otherwise $m(s_i,s_j)=2$.
(Throughout this section, we take indices mod $n$ for simple reflections and simple roots.)
A Cartan matrix $A$ for $W$ has $a_{i\,(i\pm1)}=-1$ and other off-diagonal entries $0$.

We construct a root system for~$A$ in the vector space $\reals^{n+1}$, with standard basis $\e_1,\ldots,\e_{n+1}$.
We further define vectors $\e_i$ for all $i\in\integers$ by defining $\delta$ to be $\e_{n+1}-\e_1$ and setting $\e_{i+n}=\e_i+\delta$ for all $i$.
(Since the vectors $\e_i$ are indexed by $\integers$, indices on vectors $\e$ are \emph{not} taken mod~$n$.)
We define a symmetric bilinear form $K$ on $\reals^{n+1}$ by taking the usual inner product on the linear span of $\e_1,\ldots,\e_n$ but setting $K(\e_{n+1},x)=K(\e_1,x)$ for all $x\in\reals^{n+1}$.
In particular, $K(\delta,x)=0$ for all $x\in\reals^{n+1}$.

Define vectors $\alpha_i=\e_{i+1}-\e_i$ and $\alpha_i\ck=\alpha_i$ for $i=1,\ldots,n$.
These vectors are naturally indexed modulo $n$, because ${\e_{i+n+1}-\e_{i+n}}={(\e_{i+1}+\delta)-(\e_i+\delta)}={\e_{i+1}-\e_i}$ for all $i\in\integers$.
The linear span of $\alpha_1,\ldots,\alpha_n$ is the subspace $V$ of $\reals^{n+1}$ consisting of vectors whose coordinates sum to zero.
We see that $K(\alpha_i\ck,\alpha_j)=a_{ij}$ for all $i,j\in\set{1,\ldots,n}$.
Thus we have constructed simple roots and a form $K$ corresponding to $A$.  
The vector $\delta$ is the vector described in Section~\ref{aff sec}.
That is, its simple-roots coordinates (the all-ones vector) are a $0$-eigenvector of $A$ with positive entries and gcd~$1$.

The positive roots for $A$ are the vectors of the form $\e_j-\e_i$ for $i<j\in\integers$ with $i\not\equiv j\mod n$.
Any given positive root has infinitely many expressions as a difference $\e_j-\e_i$, because $\e_j-\e_i=(\e_{j+n}-\delta)-(\e_{i+n}-\delta)=\e_{j+n}-\e_{i+n}$ and so forth.
The reflection $t$ orthogonal to a root $\e_j-\e_i$ has 
\[t(\e_k)=\begin{cases}
\e_{k-j+i}&\text{if }k\equiv j\mod n,\\
\e_{k+j-i}&\text{if }k\equiv i\mod n,\text{ or}\\
\e_k&\text{otherwise,}
\end{cases}\]
and in particular 
\[s_i(\e_j)=\begin{cases}
\e_{j+1}&\text{if }j\equiv i\mod n,\\
\e_{j-1}&\text{if }j\equiv i+1\mod n,\text{ or}\\
\e_j&\text{otherwise.}
\end{cases}\]

The set $\Phi_\fin=\set{\pm(\e_j-\e_i):1\le i<j\le n}$ is a finite root system of type $A_{n-1}$ and the full set $\Phi$ of real roots is $\set{\beta+k\delta:\beta\in\Phi_\fin,k\in\integers}$.

We describe the affine Coxeter group $W$ in terms of its action on $\set{\e_i:i\in\integers}$.
The action on indices $i\in\integers$ is the usual description of~$W$ as  the group $\Stilde_n$ of \newword{affine permutations}.
(See, e.g.\ \mbox{\cite[Section~8.3]{Bj-Br}}.)
These are permutations $\pi$ of $\integers$ such that $\pi(i+n)=\pi(i)+n$ for all $i\in\integers$ and $\sum_{i=1}^n\pi(i)=\binom{n+1}2$.
We also consider the group $\SZmodn$ of \newword{periodic permutations}: permutations $\pi$ of $\integers$ such that ${\pi(i+n)=\pi(i)+n}$ for all $i\in\integers$.
The group of periodic permutations is sometimes called the extended affine symmetric group.
An affine permutation (or more generally a periodic permutation) $\pi$ acts on $\reals^{n+1}$ by sending each $\e_i$ to~$\e_{\pi(i)}$.

A permutation $\pi\in \SZmodn$ can be described in terms of its cycle structure.
A finite cycle in $\pi$ may not contain two entries that are equivalent modulo $n$, and for every finite cycle $(a_1\,\,\,a_2\,\cdots\,a_k)$ in $\pi$, the cycle $(a_1+\ell n\,\,\,a_2+\ell n\,\cdots\,a_k+\ell n)$ is also present in $\pi$ for every $\ell\in\integers$.
We write $(a_1\,\,\,a_2\,\cdots\,a_k)_n$ for the infinite product $\prod_{\ell\in\integers}(a_1+\ell n\,\,\,a_2+\ell n\,\cdots\,a_k+\ell n)$ of cycles.
An infinite cycle in $\pi$ necessarily has two entries that are equivalent modulo $n$, and, for any entry $a$ in the cycle, the cycle is completely determined by the sequence of entries from $a$ to the next entry that is equivalent to $a$ modulo $n$.
Thus we can represent infinite cycles in $\pi$ by writing $(\cdots\,a_1\,\,\,a_2\,\cdots\,a_{k+1}\,\cdots)$, where $a_{k+1}\equiv a_1\mod n$ and $a_i\not\equiv a_1\mod n$ for $i=1,\ldots,k$.

The simple reflections in $\Stilde_n$ are $s_i=(i\,\,\,i+1)_n$ for $i=1,\ldots,n$, and similarly the reflection orthogonal to a root $\e_j-\e_i$ is $(i\,\,\,j)_n$.
Thus the set of reflections is $T=\set{(i\,\,\,j)_n:i<j,i\not\equiv j\mod n}$.

A periodic permutation $\pi$ is completely determined by the sequence of entries $\pi(1),\pi(2),\ldots,\pi(n)$, sometimes called the ``window'' of $\pi$.
The affine permutations are the periodic permutations whose window sums to $\binom{n+1}2$.
\begin{lemma}\label{window mod n}
Suppose $\pi,\tau\in \SZmodn$.
Then  
\begin{enumerate}[\quad\rm\bf1.]
\item \label{one per class}
$\set{\pi(1),\ldots,\pi(n)}$ contains exactly one representative of each mod-$n$ class.
\item \label{mod n}
$\sum_{i=1}^n\pi(i)\equiv \binom{n+1}2 \mod n$.
\end{enumerate}
\end{lemma}
\begin{proof}
If $1\le i<j\le n$ and $\pi(i)\equiv\pi(j)\mod n$, then write $\pi(i)=\pi(j)-kn$ for some $k\in\integers$.
Then $\pi(i+kn)=\pi(j)$, and this contradiction proves Assertion~\ref{one per class}.

Since $\pi\in \SZmodn$, for each $i\in\set{1,\ldots,n}$, we can write $\pi(i)$ uniquely as $\mu(i)+\kappa(i)n$ for $\mu(i),\kappa(i)\in\integers$ with $1\le\mu(i)\le n$.
Then  $\sum_{i=1}^n\pi(i)$ is equal to $\sum_{i=1}^n(\mu(i)+\kappa(i)n)$, which by Assertion~\ref{one per class} equals $\binom{n+1}2+\sum_{i=1}^n\kappa(i)n$.
That proves Assertion~\ref{mod n}.
\end{proof}

\begin{remark}\label{normal}
One can show that the map $\pi\mapsto\frac1n\left(\sum_{i=1}^n\pi(i)-\binom{n+1}2\right)$ is a surjective homomorphism from $\SZmodn$ to $\integers$, whose kernel is the subgroup $\Stilde_n$.
\end{remark}

\begin{prop}\label{SZ mod n gen}
The group $\SZmodn$ is generated by the set $\set{s_1,\ldots,s_n}\cup\set{\ell_1}$, where $\ell_1$ is the permutation whose only nontrivial cycle is  $(\cdots\,1\,\,\,\,1+n\,\cdots)$.
\end{prop}
\begin{proof}
Suppose $\pi\in \SZmodn$.
Lemma~\ref{window mod n}.\ref{mod n} says that $\sum_{i=1}^n\pi(i)=\binom{n+1}2+kn$ for some $k\in\integers$.
Writing $\pi'$ for $\ell_1^{-k}\circ\pi$, we have $\sum_{i=1}^n\pi'(i)=\binom{n+1}2$, so $\pi'\in\Stilde_n$.
Since~$\Stilde_n$ is generated by $\set{s_1,\ldots,s_n}$, we see that $\pi=\ell_1^k\pi'$ is in the group generated by $\set{s_1,\ldots,s_n}\cup\set{\ell_1}$.
\end{proof}

More generally, write $\ell_i$ for $(\cdots\,i\,\,\,\,i+n\,\cdots)$ and write $L$ for the set $\set{\ell_1,\ldots,\ell_n}\cup\set{\ell_1^{-1},\ldots,\ell_n^{-1}}$ of \newword{loops}.
The name looks forward to a map that associates $\ell_i$ or~$\ell_i^{-1}$ to a noncrossing partition of the annulus whose only nontrivial block resembles a loop at $i$.
The following is an immediate corollary of Proposition~\ref{SZ mod n gen}.
\begin{cor}\label{gen T L}
$\SZmodn$ is generated by $T\cup L$.
\end{cor}
The set $T\cup L$ has an important property (analogous to the set $T$ of generators of $\Stilde_n$):  It is closed under conjugation in $\SZmodn$.

Just as we are interested in the interval $[1,c]_T$ in the absolute order on $\Stilde_n$, we will also be interested in the analogous interval in the group $\SZmodn$, with respect to the generating set $T\cup L$.
We write $\ell_{T\cup L}$ for the length function in $\SZmodn$ relative to the generating set $T\cup L$.
We also write $\le_{T\cup L}$ for the partial order on $\SZmodn$ defined by $u\le_{T\cup L}w$ if and only if $\ell_{T\cup L}(w)=\ell_{T\cup L}(u)+\ell_{T\cup L}(u^{-1}w)$, analogous to the absolute order on $\Stilde_n$.
We write $[1,c]_{T\cup L}$ for the interval between $1$ and $c$ in this order, where the subscript specifies not only the generating set, but indirectly also the group where the interval lives.

\subsection{Coxeter elements}\label{cox A sec}
In this section, we describe the choice of a Coxeter element in $\Stilde_n$ in terms of placements of integers on the boundary of an annulus.
In Section~\ref{plane A sec}, we show how the placement of numbers on the annulus arises from the construction described in Section~\ref{aff sec}, projecting an orbit in $V$ to the Coxeter plane in $V^*$.

In type $\afftype{A}_{n-1}$, the Coxeter diagram is a cycle with each $s_i$ adjacent to $s_{i-1}$ and~$s_{i+1}$ (with indices modulo $n$ as usual).
Thus the choice of a Coxeter element $c$ is exactly the choice of $s_{i-1}\to s_i$ or $s_i\to s_{i-1}$ for each $i=1,\ldots,n$ to make an acyclic orientation.
We record the choice of $c$ as follows:
The numbers $1,\ldots,n$ are placed in clockwise order over one full turn about the center of the annulus.
The number $i$ is placed on the outer boundary if and only if $s_{i-1}\to s_i$ (i.e.\ $s_{i-1}$ precedes $s_i$ in~$c$) or on the inner boundary if and only if $s_i\to s_{i-1}$ (i.e.\ $s_i$ precedes $s_{i-1}$ in~$c$).
We identify the numbers $1,\ldots,n$ with their positions on the boundary of the annulus, and call them \newword{outer points} or \newword{inner points}.
This construction creates a bijection from the set of Coxeter elements to the set of all partitions of $\set{1,\ldots,n}$ into a nonempty set of outer points and a nonempty set of inner points.
If there are no outer points or no inner points, the corresponding orientation of the diagram is a directed cycle, and thus does not specify a Coxeter element.

\begin{example}\label{labeled annulus ex}  
Figure~\ref{labeled annulus} shows the case where $W=\Stilde_7$ and the Coxeter element is $c= s_6s_5s_2s_1s_3s_4s_7$, so that the outer points are $3,4,7$ and the inner points are $1,2,5,6$.
In cycle notation, $c=(\cdots\,3\,\,\,4\,\,\,7\,\,\,10\,\cdots)(\cdots\,6\,\,\,5\,\,\,2\,\,\,1\,\,-1\,\cdots)$.
The significance of the gray vertical line segment in the picture will be explained later.
\end{example}
\begin{figure}
\scalebox{1}{\includegraphics{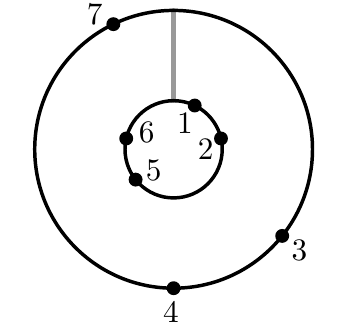}}
\caption{Inner points and outer points for $c= s_6s_5s_2s_1s_3s_4s_7$}
    \label{labeled annulus}
\end{figure}


\begin{lemma}\label{c annulus}
Let $c$ be a Coxeter element of $\Stilde_n$, represented as a partition of $\set{1,\ldots,n}$ into inner points and outer points.
If $a_1,\ldots, a_k$ are the outer points in increasing order and $b_1,\ldots,b_{n-k}$ are the inner points in decreasing order, then
\[c=(\cdots\,a_1\,\,\,a_2\,\cdots\,a_k\,\,\,a_1+n\,\cdots)(\cdots\,b_1\,\,\,b_2\,\cdots\,b_{n-k}\,\,\,b_1-n\,\cdots).\]
\end{lemma}
\begin{proof}
Since every orientation of the Coxeter diagram for $\Stilde_n$ is conjugate, by source-sink moves, to an orientation such that $s_n$ is the unique source and there is a unique sink $s_k$, we can prove the lemma in two steps.
The first step, whose details are omitted, is to compute that the lemma holds in the case where $c=s_ns_{n-1}\cdots s_{k+1}s_1s_2\cdots s_k$, 
so that the outer points are $1,\ldots,k$ and the inner points are $k+1,\ldots,n$.

The second step is to show that the conclusion of the lemma is preserved under source-sink moves.
Suppose $c$ has cycles as described in the lemma, and suppose~$s_i$ is a source.
That means that $i$ is inner and $i+1$ is outer.
After the source-sink move, $i$ is outer and $i+1$ is inner.
Conjugating by $s_i=(i\,\,\,i+1)_n$ serves to swap $i$ and $i+1$ in the cycle notation for $c$, thus preserving the conclusion of the lemma.
If $s_i$ is a sink, the proof is the same, with inner and outer reversed.
\end{proof}

\subsection{Projecting to the Coxeter plane}\label{plane A sec}
Recall that the Coxeter plane in $V^*$ is spanned by the vectors $\omega_c(\delta,\,\cdot\,)$ and $\omega_c(\gamma_c,\,\cdot\,)$.
We next explain how to project an orbit in $V$ to the Coxeter plane in $V^*$ and recover the annulus from Section~\ref{cox A sec}.

\begin{prop}\label{om del rho}
Let $c$ be a Coxeter element of $\Stilde_n$, represented as a partition of $\set{1,\ldots,n}$ into inner points and outer points.
Then 
\[\frac12\omega_c(\delta,\,\cdot\,)=\sum_{\substack{1\le i\le n\\i\text{ outer}}}(\rho_i-\rho_{i-1})=-\sum_{\substack{1\le i\le n\\i\text{ inner}}}(\rho_i-\rho_{i-1})=\sum_{\substack{1\le i\le n\\i\text{ outer}\\i+1\text{ inner}}}\rho_i-\sum_{\substack{1\le i\le n\\i\text{ inner}\\i+1\text{ outer}}}\rho_i.\]
\end{prop}
The quantities $i-1$ and $i+1$ occurring in the proposition are interpreted mod~$n$.
\begin{proof}
The coefficient of $\rho_i$ in $\omega_c(\delta,\,\cdot\,)$ is $\omega_c(\delta,\alpha_i\ck)$.
Since $\delta=\alpha_1+\cdots+\alpha_n$, Equation \eqref{omega def} and the skew-symmetry of $\omega_c$ imply that
\[\omega_c(\delta,\alpha_i\ck)=-\omega_c(\alpha_i\ck,\alpha_{i-1})-\omega_c(\alpha_i\ck,\alpha_{i+1})=\pm 1\pm1,\]
taking the ``$+$'' in the first $\pm$ if and only if $i$ is an \emph{outer} point, and taking ``$+$'' in the second $\pm$ if and only if $i+1$ an \emph{inner} point.
Thus the coefficient of $\rho_i$ in $\omega_c(\delta,\,\cdot\,)$ is $2$ if $i$ is outer and $i+1$ is inner, is $-2$ if $i$ is inner and $i+1$ is outer, or is $0$ otherwise.
We have verified the rightmost expression for $\frac12\omega_c(\delta,\,\cdot\,)$, and the other two expressions follow.
\end{proof}

The following result is not needed in the arguments that follow, but is useful in computing specific examples.
Let $\#\inn$ be the number of inner points associated to~$c$ and let $\#\out$ be the number of outer points.
Further, let ${(\#\inn_{\le k})}$ be the number of inner points less than or equal to $k$, etc.

\begin{lemma}\label{gammac}
Let $c$ be a Coxeter element of $\Stilde_n$, represented as a partition of $\set{1,\ldots,n}$ into inner points and outer points.
Then 
\[\gamma_c=\frac1n\left((\#\inn)\sum_{i\text{ outer}}\e_i-(\#\out)\sum_{i\text{ inner}}\e_i\right)=\sum_{k=1}^nb_k\alpha_k,\]
where $b_k=\frac1n[(\#\inn_{\le k})(\#\out_{>k})-(\#\inn_{>k})(\#\out_{\le k})]$.
\end{lemma}
\begin{proof}
Write $\gamma$ for $\frac1n\left((\#\inn)\sum_{i\text{ outer}}\e_i-(\#\out)\sum_{i\text{ inner}}\e_i\right)$.
The subspace $V_\fin$ of~$V$ is the span of the simple roots $\alpha_1,\ldots,\alpha_{n-1}$ (i.e.\ all simple roots except~$\alpha_n$).
This is the subspace of $\reals^{n+1}$ with $(n+1)\st$ coordinate $0$ and with coordinates summing to $0$.
Thus $\gamma$ is indeed in $V_\fin$.
We verify that it is a generalized $1$-eigenvector associated to $\delta$.
Lemma~\ref{c annulus} implies that $c\Bigl(\sum_{i\text{ outer}}\e_i\Bigr)={\Bigl(\sum_{i\text{ outer}}\e_i\Bigr)+\delta}$ and ${c\Bigl(\sum_{i\text{ inner}}\e_i\Bigr)=\Bigl(\sum_{i\text{ inner}}\e_i\Bigr)-\delta}$.
Thus $c(\gamma)=\gamma+\frac1n[(\#\inn)+(\#\out)]\delta=\gamma+\delta$, as desired.
Furthermore, 
\begin{align*}
\gamma&=\frac1n\sum_{\substack{i\text{ inner}\\j\text{ outer}}}(\e_j-\e_i)\\
&=\frac1n\sum_{\substack{i\text{ inner}\\j\text{ outer}\\i<j}}\sum_{k=i}^{j-1}\alpha_k-\sum_{\substack{i\text{ inner}\\j\text{ outer}\\j<i}}\sum_{k=j}^{i-1}\alpha_k\\
&=\frac1n\sum_{k=1}^n\alpha_k\left[(\#\inn_{\le k})(\#\out_{>k})-(\#\inn_{>k})(\#\out_{\le k})\right].\qedhere
\end{align*}
\end{proof}

Our next step is greatly simplified by working in $\reals^{n+1}$ rather than in the subspace~$V$.
Specifically, we think of $V$ as the quotient of $\reals^{n+1}$ modulo the line spanned by $\e_1+\cdots+\e_n$, so that we can name a vector in $V$ more conveniently by naming a vector in $\reals^{n+1}$.
(By checking each $s_i$, we verify that the action of $W$ fixes $\e_1+\cdots+\e_n$.
The subscript $n$ is not an error:  $\e_1+\cdots+\e_{n+1}$ is not fixed by~$W$.)
In particular, we will write expressions of the form $\br{x,y}$ for $x\in V^*$ and $y\in\reals^{n+1}$, interpreting $y$ as a vector in $V$ (or equivalently, extending the pairing $\br{\,\cdot\,,\,\cdot\,}$ to a bilinear map $V^*\times\reals^{n+1}$ by declaring $\br{x,\e_1+\cdots+\e_n}=0$ for all $x\in V^*$).

The most natural $W$-orbit in $\reals^{n+1}$ is $\set{\e_j:j\in\integers}$, on which $W$ acts by affine permutations of indices.
The following proposition will help us project this orbit to the Coxeter plane.

\begin{prop}\label{rho e}
Let $i\in\set{1,\ldots,n}$ and $j\in\integers$.
Then \[\br{\rho_i,\e_j}=\frac{i}{n}+\left\lfloor\frac{j-i-1}n\right\rfloor.\]
\end{prop}
\begin{proof}
It is straightforward to verify that $\sum_{i=1}^n\frac{i}{n}\alpha_i=\e_{n+1}-\frac1{n}(\e_1+\cdots+\e_n)$.
Thus $\br{\rho_i,\e_{n+1}}=\frac in$.
Since the $\rho_i$ are dual basis to the $\alpha_i\ck$, and since (in type $\afftype{A}$) $\alpha_i\ck=\alpha_i$ for all $i=1,\ldots,n$, we conclude that the proposition holds in the case $j=n+1$.
To complete the proof, we show that the proposition holds for some $j\in\integers$ if and only if it holds for $j+1$.
Since $\e_{j+1}=\e_j+\alpha_j$ (with the index on $\alpha_j$ interpreted modulo $n$), we see that 
\[\br{\rho_i,\e_{j+1}}=\begin{cases}
\br{\rho_i,\e_j}+1&\text{if }i\equiv j\mod n,\text{ or}\\
\br{\rho_i,\e_j}&\text{otherwise}.\\
\end{cases}\]
The proposed formula for $\br{\rho_i,\e_j}$ exhibits the same behavior as $j$ is replaced by $j+1$.
\end{proof}

\begin{prop}\label{om del e}
Let $c$ be a Coxeter element of $\Stilde_n$, represented as a partition of $\set{1,\ldots,n}$ into inner points and outer points.
Then $\frac12\omega_c(\delta,\e_j)$ takes two distinct values (differing by $1$ and having strictly opposite signs), as $j$ ranges over $\integers$, taking the positive value if and only if $j\mod n$ is an inner point.
\end{prop}
\begin{proof}
Propositions~\ref{om del rho} and~\ref{rho e} imply that 
\[\frac12\omega_c(\delta,\e_j)=\sum_{\substack{i\text{ outer}\\i+1\text{ inner}}}\left(\frac{i}n+\left\lfloor\frac{j-i-1}n\right\rfloor\right)-\sum_{\substack{i\text{ inner}\\i+1\text{ outer}}}\left(\frac{i}n+\left\lfloor\frac{j-i-1}n\right\rfloor\right).\]
This is $\sum_{\substack{i\text{ outer}\\i+1\text{ inner}}}\left\lfloor\frac{j-i-1}n\right\rfloor-\sum_{\substack{i\text{ inner}\\i+1\text{ outer}}}\left\lfloor\frac{j-i-1}n\right\rfloor$ plus a quantity that is independent of $j$.  
Thus as $j$ runs through $\integers$ in increasing order, $\frac12\omega_c(\delta,\e_j)$ increases by $1$ every time $j\mod n$ becomes an inner point, and decreases by $1$ every time $j\mod n$ becomes an outer point.

To complete the proof, it remains to show that $0<\frac12\omega_c(\delta,\e_j)<1$ for some $j$ such that $j\mod n$ is inner or that $-1<\frac12\omega_c(\delta,\e_j)<0$ for some $j$ such that $j\mod n$ is outer.
We will estimate $\frac12\omega_c(\delta,\e_{n+1})$.
We have 
\[\frac12\omega_c(\delta,\e_{n+1})=\sum_{\substack{i\text{ outer}\\i+1\text{ inner}}}\left(\frac{i}n+\left\lfloor\frac{n-i}n\right\rfloor\right)-\sum_{\substack{i\text{ inner}\\i+1\text{ outer}}}\left(\frac{i}n+\left\lfloor\frac{n-i}n\right\rfloor\right).\]
Since these sums are over $1\le i\le n$, all of the floors are zero, so 
\[\frac12\omega_c(\delta,\e_{n+1})=\sum_{\substack{i\text{ outer}\\i+1\text{ inner}}}\frac{i}n-\sum_{\substack{i\text{ inner}\\i+1\text{ outer}}}\frac{i}n.\]
Both sums in the above expression have the same number of terms, because the elements of the cycle $(1\,\,2\,\cdots\,n)$ switch from outer to inner the same number of times they switch from inner to outer.
If $1$ is outer, then each term in the first sum is smaller than the corresponding term in the second sum, so that $\frac12\omega_c(\delta,\e_{n+1})<0$.
But also, the first term in the second sum is smaller than the second term in the first sum, and so forth.
Thus $\frac12\omega_c(\delta,\e_{n+1})$ is greater than the first term in the first sum minus the last term in the second sum, which is at least $\frac{1-n}n>-1$.
On the other hand, if $1$ is inner, then each term in the first sum is larger than the corresponding term in the second sum, so that $\frac12\omega_c(\delta,\e_{n+1})>0$.
Also, the first term in the first sum is smaller than the second term in the first sum, etc., so $\frac12\omega_c(\delta,\e_{n+1})$ is less than the last term in the first sum minus the first term in the second sum, which is at most $\frac{n-1}n<1$.
\end{proof}

The Coxeter plane in $V^*$ is spanned by the $1$-eigenvector $\omega_c(\delta,\,\cdot\,)$ and the generalized $1$-eigenvector $\omega_c(\gamma_c,\,\cdot\,)$.
Thus the projection of a point $x\in V$ to the Coxeter plane in $V^*$ is the point $\omega_c(\gamma_c,x)\omega_c(\gamma_c,\,\cdot\,)+\omega_c(\delta,x)\omega_c(\delta,\,\cdot\,)\in V^*$.

\begin{theorem}\label{a orb proj}
Let $c$ be a Coxeter element of $\Stilde_n$, represented as a partition of $\set{1,\ldots,n}$ into inner points and outer points, and consider the  projection of the orbit $\set{\e_j:j\in\integers}$ to the Coxeter plane in $V^*$.
\begin{enumerate}[\quad\rm\bf1.]
\item \label{2 vert}
The projection takes $\set{\e_j:j\in\integers}$ into two parallel lines, each defined by a constant $\omega_c(\delta,\,\cdot\,)$-coordinate.
\item \label{pos neg vert}
One of the two lines has negative $\omega_c(\delta,\,\cdot\,)$-coordinate and contains the image of $\set{\e_j:j\mod n\text{ is outer}}$, while the other line has positive $\omega_c(\delta,\,\cdot\,)$-coordinate and contains the image of $\set{\e_j:j\mod n\text{ is inner}}$.
\item \label{bigger vert}
If $j<j'$ and $j$ and $j'$ are either both outer or both inner, then the projection of $\e_{j'}$ has strictly larger $\omega_c(\gamma_c,\,\cdot\,)$-coordinate than the projection of $\e_j$.
\item \label{evenly}
On the line containing outer points, the space between adjacent points is $-\omega_c(\delta,\e_j)$ for any $j$ outer.
On the line containing inner points, the space between adjacent points is $\omega_c(\delta,\e_j)$ for any $j$ inner.
\item \label{const dist}
The difference between the $\omega_c(\gamma_c,\,\cdot\,)$-coordinates of the projections of $\e_j$ and $\e_{j+n}$ is $\omega_c(\gamma_c,\delta)$, independent of $j$.
\end{enumerate}
\end{theorem}
\begin{proof}
Proposition~\ref{om del e} implies immediately that the inner and outer points project into two lines as in Assertions~\ref{2 vert} and~\ref{pos neg vert}.
The difference described in Assertion~\ref{const dist} is $\omega_c(\gamma_c,\e_j+\delta)-\omega_c(\gamma_c,\e_j)=\omega_c(\gamma_c,\delta)$.

To prove Assertions~\ref{bigger vert} and~\ref{evenly}, we first evaluate $\omega_c(\gamma_c,c^{-1}\e_j)$.
This is equal to $\omega_c(c\gamma_c,\e_j)$, and since $c\gamma_c=\gamma_c+\delta$, it further equals $\omega_c(\gamma_c+\delta,\e_j)$.
Thus $\omega_c(\gamma_c,c^{-1}\e_j)-\omega_c(\gamma_c,\e_j)$ is $\omega_c(\delta,\e_j)$.
If $j\mod n$ is outer, then Lemma~\ref{c annulus} says that $c^{-1}\e_j$ is the next smallest integer that is outer modulo $n$.
Also Proposition~\ref{om del e} says that $\omega_c(\delta,\e_j)$ is negative, and we conclude that $\e_j$ projects to a strictly larger $\omega_c(\gamma_c,\,\cdot\,)$-coordinate than the next smallest outer integer.
Similarly, if $j\mod n$ is inner, then Lemma~\ref{c annulus} says that $c^{-1}\e_j$ is the next largest integer that is inner (mod~$n$).
In this case, Proposition~\ref{om del e} says that $\omega_c(\delta,\e_j)$ is positive, so $\e_j$ projects to a strictly smaller $\omega_c(\gamma_c,\,\cdot\,)$-coordinate than the next largest inner integer.
\end{proof}

\begin{example}\label{proj ex}
Continuing Example~\ref{labeled annulus ex}, when $W=\Stilde_7$ and $c= s_6 s_5 s_2 s_1 s_3 s_4s_7$, we compute that $7\omega_c(\delta,\e_j)$ is $-4$ or $3$, taking the positive value if and only if $j\mod 7$ is inner, that $7\omega_c(\gamma_c,\e_j)-\frac17$ takes values $-8,-2,-7,1,4,10,9$ for $j=1,2,3,4,5,6,7$, and that $\omega_c(\gamma_c,\e_{j+7})=\omega_c(\gamma_c,\e_j)+\frac{24}7$ for all $j$.
The projection of part of the orbit $\set{\e_j:j\in\integers}$ to the Coxeter plane in $V^*$ is illustrated in Figure~\ref{proj fig}, with the $\omega_c(\gamma_c,\,\cdot\,)$-direction on the horizontal axis and the $\omega_c(\delta,\,\cdot\,)$-direction on the vertical axis.
\end{example}

\begin{figure}

\includegraphics{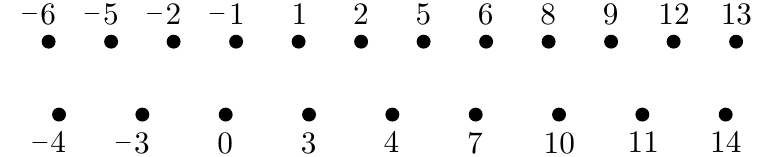}
\caption{Projecting an orbit to the Coxeter plane in $V^*$}
\label{proj fig}
\end{figure}

%
%

Theorem~\ref{a orb proj} suggests that, after projecting the orbit $\set{\e_j:j\in\integers}$ to the Coxeter plane in $V^*$, we can profitably consider the $\omega_c(\gamma_c,\,\cdot\,)$-direction in the Coxeter plane modulo the difference described in Theorem~\ref{a orb proj}.\ref{const dist}.
When we do so, instead of mapping the orbit into the two boundary lines of an infinite strip, we map the orbit into the two boundary circles of a cylinder.
We flatten the cylinder to an annulus in such a way that the line of projected points with negative $\omega_c(\delta,\,\cdot\,)$-coordinate becomes the outer boundary of the annulus, while the other line becomes the inner boundary.
Each mod-$n$ class $\set{\e_{i+kn}:k\in\integers}$ in the orbit maps to a single point on the boundary of the annulus.
We recover (up to shifting the points along the boundary without moving points through each other) the annulus with outer and inner points that was described in Section~\ref{cox A sec}.
In particular, the projection in Figure~\ref{proj fig} becomes the annulus in Figure~\ref{labeled annulus}.

\subsection{Noncrossing partitions of an annulus}\label{nc ann sec}
Fix a Coxeter element $c$ in $\Stilde_n$, encoded by a choice of inner and outer points on an annulus as described in Section~\ref{cox A sec}.
We refer to this annulus as $A$.
The inner and outer points are collectively called \newword{numbered points}.

The pair consisting of $A$ and the set of numbered points is an example of a marked surface in the sense of \cite{cats1,cats2} (where the numbered points were called marked points).
We now quote from \cite{surfnc} some definitions and results on noncrossing partitions of marked surfaces, in the special case of noncrossing partitions of the annulus $A$.
We emphasize, however, that these results appeared in \cite{BThesis} in the annulus case and were generalized to marked surfaces in~\cite{surfnc}.

A \newword{boundary segment} of $A$ is the portion of the boundary of $A$ between two adjacent numbered points.
A boundary segment may have both endpoints at the same numbered point if that is the only numbered point on one component of the boundary of $A$.

Two subsets of $A$ are related by \newword{ambient isotopy} (or simply \newword{isotopy}) if they are related by a homeomorphism from $A$ to itself, fixing the boundary $\partial A$ pointwise and homotopic to the identity by a homotopy that fixes $\partial A$ pointwise at every step.

An \newword{arc} in $A$ is a non-oriented curve in $A$, with endpoints at numbered points (possibly coinciding), that does not intersect itself except possibly at its endpoints and, except for its endpoints, does intersect the boundary of $A$.
We exclude the possibility that an arc bounds a monogon in $A$ (and thus is contractible to its endpoint) and the possibility that an arc combines with a boundary segment to bound a digon in $A$ (and thus can be deformed to coincide with the boundary segment).
We consider arcs up to ambient isotopy.  
Thus an arc is uniquely identified by which numbered points it connects and which direction and how many times the arc wraps around the hole of the annulus.

An \newword{embedded block} in $A$ is one of the following: 
\begin{itemize}
\item
a \newword{trivial block}, meaning a singleton consisting of a numbered point in $A$;
\item
an arc or boundary segment in $A$; 
\item
a \newword{disk block}, meaning a closed disk in $A$ whose boundary is a union of arcs and/or boundary segments of $A$;
\item
a \newword{dangling annular block}, meaning a closed annulus in $A$ with one component of its boundary a union of arcs and/or boundary segments of $A$ and the other component of its boundary a circle in the interior of $A$ that can't be contracted in $A$ to a point; or
\item
a \newword{nondangling annular block}, meaning a closed annulus in $A$ with each component of its boundary a union of arcs and/or boundary segments of~$A$.
\end{itemize}
The first two types of embedded blocks are \newword{degenerate disk blocks}.  The last two types of embedded blocks are referred to less specifically as \newword{annular blocks}.
Embedded blocks are considered up to ambient isotopy.

A \newword{noncrossing partition} of $A$ is a collection $\P=\set{E_1,\ldots,E_k}$ of disjoint embedded blocks (the \newword{blocks of~$\P$}) such that every numbered point is contained in some $E_i$ and at most one $E_i$ is an annular block.
Noncrossing partitions are considered up to ambient isotopy.
Thus a noncrossing partition of $A$ is a partition of the set of numbered points together with the additional data of how each block is embedded as a point, curve, disk, or annulus.
We refer to different, but isotopic, choices of the embedded blocks $E_1,\ldots,E_k$ as different \newword{embeddings} of the same noncrossing partition. 

Given that the blocks are pairwise disjoint, restricting to at most one annular block serves only to rule out the case of two dangling annular blocks, one containing numbered points on its inner boundary and the other containing numbered points on its outer boundary.

\begin{example}\label{some ncs}  
We can now give more context to Figure~\ref{nc exs}.
These are noncrossing partitions of the annulus from Example~\ref{labeled annulus ex}.
(Compare Example~\ref{proj ex}.)
Degenerate blocks are shown with some thickness, to make them visible.
\end{example}

We define a partial order on noncrossing partitions of $A$, called (in light of Theorem~\ref{A tilde main} below) the \newword{noncrossing partition lattice}:
Noncrossing partitions $\P$ and~$\Q$ have $\P\le\Q$ if and only if there exist embeddings of $\P$ and $\Q$ such that every block of $\P$ is contained in some block of $\Q$.
We write $\tNCAc$ for the set of noncrossing partitions of $A$ with this partial order.
The superscript $A$ refers to the annulus $A$ and also foreshadows the connection to affine Coxeter groups of type $\afftype{A}$.

\begin{example}\label{le ex}
Continuing Example~\ref{some ncs}, we see that the left noncrossing partition shown in Figure~\ref{nc exs} is less than the middle noncrossing partition, but that is the only order relation among the three noncrossing partitions shown.
\end{example}

The following is a result of \cite{BThesis} that we quote here as \cite[Theorem~4.3]{surfnc}, which is in turn a specialization of \cite[Theorems 2.11,~2.12]{surfnc}.

\begin{theorem}\label{A tilde main}
The poset $\tNCAc$ of noncrossing partitions of an annulus with marked points on both boundaries, $n$ marked points in all, is a graded lattice, with rank function given by $n$ minus the number of non-annular blocks.
\end{theorem}

The proof of Theorem~\ref{A tilde main} uses the \newword{curve set} of a noncrossing partition.
The curve set of an embedded block is the set of all arcs and boundary segments that (up to isotopy) are contained in the block.
The curve set $\curve(\P)$ of a noncrossing partition $\P$ is the union of the curve sets of its blocks.
The proof of Theorem~\ref{A tilde main} in \cite{surfnc} or~\cite{BThesis} constructs $\P\meet\Q$ as the unique noncrossing partition whose curve set is $\curve(\P)\cap\curve(Q)$ and appeals to the following proposition, which is a specialization of \cite[Propositions 2.16,~2.17]{surfnc}.

\begin{prop}\label{curve set}
A noncrossing partition of $A$ is determined uniquely (up to isotopy) by its curve set.
Two noncrossing partitions $\P$ and $\Q$ have $\P\le\Q$ if and only if $\curve(\P)\subseteq\curve(\Q)$.
\end{prop}

It is easy to reconstruct $\P$ from its curve set.
If any numbered point $i$ has a curve in $\curve(\P)$ connecting $i$ to itself, there is an annular block in $\P$, containing every such $i$.
There are only finitely many curves involving numbered points not in the annular block.  
These define the (possibly degenerate) disk blocks of $\P$.

Any arc or boundary segment can be thought of as a noncrossing partition, namely the noncrossing partition whose only nontrivial block is that curve or the dangling annular block one of whose boundary components is that curve.
The following theorem is \cite[Corollary~2.34]{surfnc}, where we abuse notation by thinking of $\curve(\P)$ as a set of noncrossing partitions.
\begin{theorem}\label{join curve}
Suppose $\P\in\tNCAc$.
Then $\P=\Join\curve(\P)$.
\end{theorem}

We also quote a result that describes cover relations in $\tNCAc$.  
Given $\P\in\tNCAc$, a \newword{simple connector for $\P$} is an arc or boundary segment $\kappa$ in the annulus that is not in $\curve(\P)$ and that starts in some block $E$ of $\P$, leaves $E$, and is disjoint from all other blocks until it enters some block of $E'$ of $\P$, which it does not leave again.  (Possibly $E=E'$.)

Given a simple connector $\kappa$ for $\P$, we construct the \newword{augmentation of $\P$ along~$\kappa$}, written $\P\cup\kappa$.
We first replace $E$ and $E'$ by the smallest embedded block containing $E$, $E'$ and $\kappa$.
This is constructed as follows:
If $E$ and $E'$ are distinct points, then this smallest block is $\kappa$.
If $E=E'$ is a point, then the smallest block is an annulus bounded by $\kappa$ and the noncontractable circle in the interior of~$A$.
Otherwise, we replace $E$ and $E'$ by the union of $E$ and $E'$ with a ``thickened'' version of $\kappa$ (a suitably chosen digon containing $\kappa$).
If $E$ and/or $E'$ is an curve, they might also need to be ``thickened''.
This ``thickened union'' may fail to be an embedded block because one or more components of its boundary might fail to be a union of arcs/boundary segments of $A$.
Specifically, if some component is composed of arcs and boundary segments and some other curves connecting marked points, such that each of these other curves, together with a boundary segment, bounds a digon, then we attach all such digons.
Once we have created this block containing $E$, $E'$ and $\kappa$, the result may fail to be a noncrossing partition because it may contain two annular blocks. 
In this case, both annular blocks are dangling, with one containing numbered points on the inner boundary of $A$ and one containing numbered points on the outer boundary.
We replace the two dangling annular blocks with a single non-dangling annular blocks containing these inner and outer numbered points and no others.

We use the construction of augmentation to describe cover relations in $\tNCAc$.
(The symbol ``$\covered$'' denotes a cover relation.)
The following is \cite[Proposition~2.24]{surfnc}.

\begin{prop}\label{conn edge cov}
If $\P,\Q\in\tNCAc$, then $\P\covered \Q$ if and only if there exists a simple connector $\kappa$ for $\P$ such that $\Q=\P\cup\kappa$.
\end{prop}

\begin{example}\label{cov ex}
Figure~\ref{cov fig} shows noncrossing partitions $\Q_1$, $\Q_2$, $\Q_3$, and $\Q_4$ with $\Q_1\covered\Q_2\covered\Q_3\covered\Q_4$.
(The noncrossing partitions $\Q_1$ and $\Q_4$ were also pictured in Figure~\ref{nc exs}.)
In each $\Q_i$ for $i=1,2,3$, a simple connector $\kappa_i$ is also shown (as a dashed curve) such that $\Q_{i+1}=\Q_i\cup\kappa_i$.
\end{example}

\begin{figure}
\begin{tabular}{cccc}
\scalebox{0.95}{\includegraphics{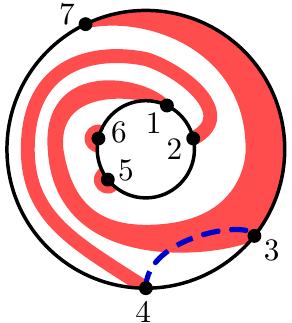}}
&\scalebox{0.95}{\includegraphics{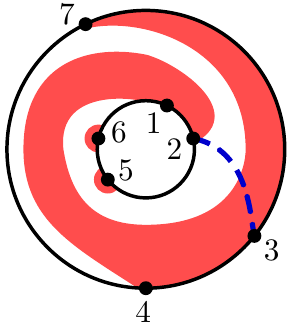}}
&\scalebox{0.95}{\includegraphics{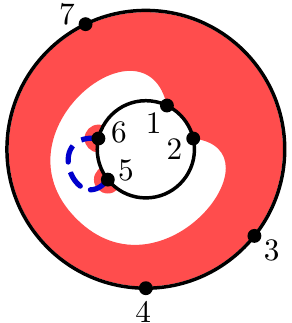}}
&\scalebox{0.95}{\includegraphics{ex1.pdf}}\\
$\Q_1$&$\Q_2$&$\Q_3$&$\Q_4$
\end{tabular}
\caption{$\Q_1\covered\Q_2\covered\Q_3\covered\Q_4$, with simple connectors}
\label{cov fig}
\end{figure}

Suppose $\P\in\tNCAc$ and $\lambda$ is a non-self-intersecting curve contained in some block $E$ of $\P$, with each endpoint at a non-numbered point on the boundary of $E$.
Then $\lambda$ is a \newword{cutting curve} for $\P$ if it either (when $E$ is annular or non-annular) cuts $E$ into two pieces, each having at least one numbered point, or (when $E$ is annular), cuts $E$ into one non-annular piece.
We allow the degenerate case where $E$ is an arc or boundary segment and $\lambda$ is a non-numbered point in $E$.
Write $\P-\lambda$ for the noncrossing partition obtained from $\P$ by cutting $E$ in this way.
More precisely, each component of the complement of $\lambda$ in $E$ contains a unique largest embedded block, and we replace $E$ with the largest block(s) in this (these) component(s).
The following proposition is immediate from Proposition~\ref{conn edge cov}.

\begin{prop}\label{cut curve cov}
If $\P,\Q\in\tNCAc$, then $\P\covered \Q$ if and only if there exists a cutting curve $\lambda$ for $\Q$ such that $\P=\Q-\lambda$.
\end{prop}

\begin{example}\label{cov ex 2}
Continuing Example~\ref{le ex}, Figure~\ref{cov fig 2} shows the same noncrossing partitions $\Q_1\covered\Q_2\covered\Q_3\covered\Q_4$ pictured in Figure~\ref{cov fig}.
In each $\Q_i$ for $i=2,3,4$, a cutting curve $\lambda_i$ is shown (as a dotted curve) such that $\Q_{i-1}=\Q_i-\lambda_i$.
\end{example}

\begin{figure}
\begin{tabular}{cccc}
\scalebox{0.95}{\includegraphics{ex3.pdf}}
&\scalebox{0.95}{\includegraphics{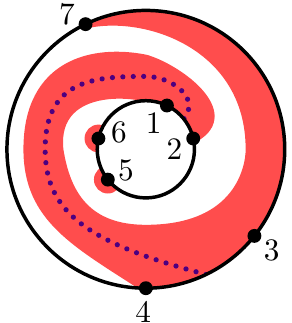}}
&\scalebox{0.95}{\includegraphics{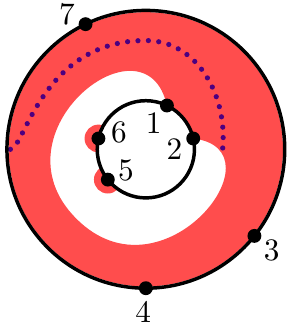}}
&\scalebox{0.95}{\includegraphics{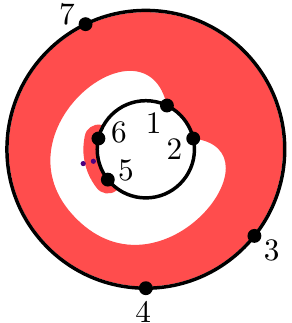}}\\
$\Q_1$&$\Q_2$&$\Q_3$&$\Q_4$
\end{tabular}
\caption{$\Q_1\covered\Q_2\covered\Q_3\covered\Q_4$, with cutting curves}
\label{cov fig 2}
\end{figure}

Write $\tNCAcircc$ for the subposet of $\tNCAc$ consisting of noncrossing partitions of~$A$ with no dangling annular blocks.

\begin{prop}\label{cov circ}
Suppose $\P,\Q\in\tNCAcircc$.
Then $\Q$ covers $\P$ in $\tNCAcircc$ if and only if $\Q$ covers $\P$ in $\tNCAc$.
If $\P\covered\Q$ and $\kappa$ is a simple connector such that $\Q=\P\cup\kappa$ and if $\kappa$ connects a block $E$ of $\P$ to itself, then $E$ is a disk or arc containing numbered points on both boundaries of $A$.
\end{prop}
\begin{proof}
If $\P\covered\Q$ in $\tNCAc$, then $\P\covered\Q$ in $\tNCAcircc$, because $\tNCAcircc$ is an induced subposet.
Conversely, suppose $\P\covered\Q$ in $\tNCAcircc$, and suppose for the sake of contradiction that there exists $\P'\in\tNCAc$ such that $\P<\P'<\Q$.
Then $\P'\not\in\tNCAcircc$, so it has a dangling annular block $E'$.
Since $\Q\in\tNCAcircc$, $E'$ is contained in some non-dangling annular block $E$ of $\Q$.
Let $\kappa$ be some arc or boundary segment connecting a numbered point $x$ in $E'$ to a numbered point $y$ in $E$ on the opposite boundary of the annulus.
Certainly $\kappa\not\in\curve(\P)$, because the block of $\P$ containing $x$ is contained in $E'$, which does not contain $y$.
On the other hand, $\kappa\in\curve(\Q)$ because it is contained in $E$.
The curve $\kappa$ might not be a simple connector for $\P$, but we can carry out the construction of $\P\cup\kappa$, which has no annular block because $\P$ has no annular block and $\kappa$ connects disk blocks to disk blocks.
In particular, $\P\cup\kappa$ is in $\tNCAcircc$ because is has no annular block.
Also, $\P<\P\cup\kappa$ because $\kappa\not\in\curve(\P)$, and $\P\cup\kappa<\Q$ because $\P\cup\kappa$ has no annular block.
This contradiction to the supposition that $\P\covered\Q$ in $\tNCAcircc$ shows that $\P\covered\Q$ in $\tNCAc$.

The second assertion of the proposition is immediate from the fact that $\Q$ is in $\tNCAcircc$.
\end{proof}

The poset $\tNCAcircc$ (noncrossing partitions of the annulus with no dangling annular blocks) can fail to be a lattice, as illustrated in the following example.

\begin{example}\label{circ fail}
Take $n=4$ and $c=s_4s_3s_1s_2$.
\begin{figure}
\begin{tabular}{ccccc}
\includegraphics{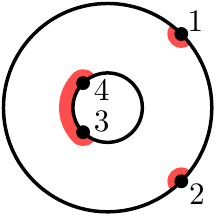}
&\includegraphics{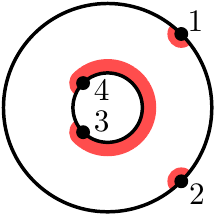}
&\includegraphics{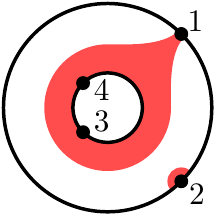}
&\includegraphics{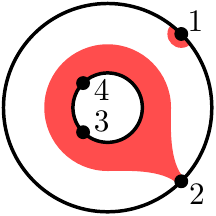}
&\includegraphics{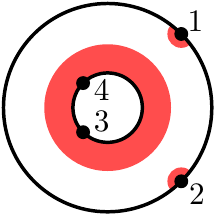}\\
$\P$&$\Q$&&&$\P\join\Q$
\end{tabular}
\caption{Joins in $\protect\tNCAc$ may need dangling annular blocks}
\label{circ fail fig}
\end{figure}
The left two pictures in Figure~\ref{circ fail fig} show noncrossing partitions $\P,\Q\in\tNCAcircc$.
The next two pictures show minimal upper bounds for $\P$ and $\Q$ in $\tNCAcircc$.
The rightmost picture shows the join of $\P$ and $\Q$ in $\tNCAc$, which is not in $\tNCAcircc$ because it has a dangling annular block.  
\end{example}

\subsection{Isomorphisms}\label{isom sec}
We now define a map $\perm:\tNCAc\to \SZmodn$ and prove that it is an isomorphism to $[1,c]_{T\cup L}$.
Given $\P\in\tNCAc$, we define $\perm(\P)$ by reading the cycle notation of a permutation from the embedding of the partition in the annulus.  
Each component of the boundary of each block is read as a cycle by following the boundary, keeping the interior of the block on the right.
(Each degenerate disk block is first thickened to a disk with one or two numbered points on its boundary, and is thus read as a singleton cycle or transposition.) 
The entries in the cycle are the numbered points encountered along the boundary, but adding a multiple of $n$ to each, with the multiple of $n$ increasing or decreasing according to how the block wraps around the annulus.
Specifically, the multiple of $n$ changes when we cross the \newword{date line}, a radial line segment separating $n$ and~$1$, shown as the gray segment in Figure~\ref{labeled annulus}.
When a numbered point is reached, we record its value (between $1$ and $n$) plus $wn$, where $w$ is the number of times we have crossed the date line clockwise minus the number of times we have crossed counterclockwise, since starting to read the cycle.
When we return to the numbered point where we started, if we have added or subtracted a nonzero multiple of $n$, then we continue around the boundary again, obtaining an infinite cycle.
Otherwise, if we have recorded the values $a_1,a_2,\ldots,a_k$ around the boundary, we obtain the infinite product $(a_1\,\,\,a_2\,\cdots\,a_k)_n$ of cycles.
More specifically, a disk block becomes such an infinite product of finite cycles, a dangling annular block becomes a single infinite cycle (either increasing or decreasing), and a non-dangling annular block becomes a pair of infinite cycles, one increasing and one decreasing.

\begin{example}\label{perm ex}
Labeling the noncrossing partitions shown in Figure~\ref{nc exs} (Example~\ref{some ncs}) as $\P_1,\P_2,\P_3$ from left to right, we apply $\perm$ to obtain the following permutations in $\SZmod7$:
\begin{align*}
\perm(\P_1)&=(1\,\,-7\,\,-4)_7\,(2\,\,-3)_7\,(5)_7\,(6)_7\\
\perm(\P_2)&=(\cdots\,1\,\,-5\,\,-6\,\cdots)\,(\cdots\,3\,\,\,4\,\,\,7\,\,\,10\,\cdots)\,(5\,\,\,6)_7\\
\perm(\P_3)&=(1\,\,-1\,\,-2)_7\,(2)_7\,(3)_7\,(\cdots\,4\,\,\,7\,\,\,11\,\cdots)
\end{align*} 
\end{example}

We will prove the following theorem.

\begin{theorem}\label{isom}
The map $\perm:\tNCAc\rightarrow \SZmodn$ is an isomorphism from $\tNCAc$ to the interval $[1,c]_{T\cup L}$ in $\SZmodn$.   
It restricts to an isomorphism from $\tNCAcircc$ to the interval $[1,c]_T$ in $\Stilde_n$.
\end{theorem}

As a first step in the proof, we observe that the cycle structure of $\perm(\P)$ uniquely determines $\P$:
A single infinite cycle uniquely specifies a dangling annular block, two infinite cycles uniquely specify a non-dangling annular block, and a family of mod-$n$ translates of a finite cycle uniquely specifies a disk.
Thus we have one piece of the proof of Theorem~\ref{isom}:

\begin{prop}\label{one to one}
The map $\perm$ is one-to-one.
\end{prop}

To continue the proof of Theorem~\ref{isom}, we need a lemma about multiplying permutations in $\SZmodn$ by reflections and loops.

\begin{lemma}\label{only one}
If $\pi\in\SZmodn$ and $\tau$ is a reflection or loop, then $\tau\pi$ has at most one more \textup{mod-}$n$ class of finite cycles than $\pi$ has.
\end{lemma}
\begin{proof}
We run through all the possibilities for how $\tau$ relates to $\pi$.
Recall that the set of reflections is $\set{(a\,\,\,b)_n:a\not\equiv b\mod n}$ and the set of loops is 
\[\set{(\cdots\,a\,\,\,\,a+n\,\cdots):a=1,\ldots,n}\cup\set{(\cdots\,a\,\,\,\,a-n\,\cdots):a=1,\ldots,n}.\]

If $\tau=(a\,\,\,b)_n$ with $a$ and $b$ in the same finite cycle of $\pi$, then we can write the class of cycles as $(a_1 \,\,\, a_2\,\cdots\,a_k)_n$ and write $\tau$ as $(a_1 \,\,\, a_i)_n$ with $1<i\le k$.
Then $\tau\pi$ has cycles $(a_1 \,\,\, a_2\,\cdots\,a_{i-1})_n$ and $(a_i\,\,\,a_{i+1}\cdots\,a_k)_n$.

If $\tau=(a\,\,\,b)_n$ with $a$ and $b$ in the same infinite cycle of $\pi$, then we can write the cycle as $(\cdots\,a_1\,\cdots\,a_k\,\,\,a_1+qn\,\cdots)$ for some $q\neq0$ and take $a=a_1$ and $b=a_i+rqn$ with $1<i\le k$ and $r\in\integers$.
If $r=0$, then $\tau\pi$ has a new finite cycles $(a_1\,\cdots\,a_{i-1})_n$ and an infinite cycle ${(\cdots\,a_i\,\cdots\,a_k\,\,\,a_i+qn\,\cdots)}$.
If $r=-1$, then $\tau\pi$ has an infinite cycle $(a_1\,\cdots\,a_{i-1}\,\,\,a_1+qn\,\cdots)$ and a new finite cycle $(a_i\,\cdots\,a_k)$.
Otherwise, $\tau\pi$ has an infinite cycle $(a_1\,\cdots\,a_{i-1}\,\,\,a_1-rqn\,\cdots)$ and an infinite cycle $(\cdots\,a_i\,\,\,a_{i+1}\cdots\,a_k\,\,\,a_i+(r+1)qn\,\cdots)$, but no new finite cycles are created.

If $\tau=(a\,\,\,b)_n$ with $a$ and $b$ in different finite cycles of $\pi$, then either the two cycles are in the same class or not.
If they are in the same class, then we can write the class of cycles as $(a_1 \,\,\, a_2 \,\cdots\, a_k)_n$ and write $\tau=(a_1\,\,\,a_i+qn)_n$ with $1<i\le k$ and $q\neq0$.
In $\tau \pi$, the infinite class of finite cycles becomes two finite classes of infinite cycles including $(\cdots\,a_1\,\cdots\,a_{i-1}\,\,\,a_1-qn\,\cdots)$ and $(\cdots\,a_i\,\,\,a_{i+1}\,\cdots\,\,a_k\,\,\,a_i+qn\,\cdots)$.
If the two cycles are in different classes, then we can write the classes as $(a_1 \,\,\, a_2 \,\cdots\, a_k)_n$ and $(b_1 \,\,\, b_2\,\cdots\,b_m)_n$ and write $\tau=(a_1 \,\,\, b_1)_n$.
In $\tau \pi$, those two classes of finite cycles combine into one class $(a_1 \,\,\, a_2\,\cdots\,a_k \,\,\, b_1 \,\,\, b_2\,\cdots\,b_m)_n$. 

If $\tau=(a\,\,\,b)_n$ with $a$ and $b$ in different infinite cycles of $\pi$, then either the two cycles are in the same class or not.
If they are in the same class, then we can write the cycle containing $a$ as $(\cdots\,a_1\,\cdots\,a_k\,\,\,a_1+qn\,\cdots)$ and write $\tau=(a_1\,\,\,a_i+rn)_n$ with $1<i\le k$ and $r\not\equiv0\mod q$.
Then $\tau\pi$ has cycles $(\cdots\,a_1\,\cdots\,a_{i-1}\,\,\,a_1-rn\,\cdots)$ and $(\cdots\,a_i\,\cdots\,a_k\,\,\,a_i+(q+r)n\,\cdots)$, which are both infinite.
If they are not in the same class, then we can write the cycles as $(\cdots\,a_1\,\cdots\,a_k\,\,\,a_1+qn\,\cdots)$ and $(\cdots\,b_1\,\cdots\,b_m\,\,\,b_1+rn\,\cdots)$ and write $\tau=(a_1 \,\,\, b_1)_n$.
In $\tau\pi$, there is a cycle containing the sequence $a_1\,\cdots\,a_k\,\,\,b_1+qn\,\cdots\,b_m+qn\,\,\,a_1+(q+r)n$.
If $q+r=0$, then this is a finite cycle.
If $q+r\neq0$, then this is a infinite cycle.

If $\tau=(a\,\,\,b)_n$ with $a$ in a finite cycle of $\pi$ and $b$ in an infinite cycle of $\pi$, then we can write the class of the finite cycle as $(a_1 \,\,\, a_2\,\cdots\,a_k)_n$, write the infinite cycle as $(\cdots\,b_1\,\cdots\,b_m\,\,\,b_1+qn\,\cdots)$ and write $\tau=(a_1\,\,\,b_1)_n$.
In $\tau\pi$, these cycles combine into one infinite cycle $(\cdots\,a_1\,\cdots\,a_k\,\,\,b_1\,\cdots b_m \,\,\, a_1+qn\,\cdots)$.

Suppose $\tau=(\cdots\,a\,\,\,a\pm n\,\cdots)$.
If $a$ is in a finite cycle of $\pi$, then write the class of finite cycles as $(a_1\,\cdots\,a_k)_n$ and take $a=a_1$.
Then $\tau\pi$ has, in place of that class of finite cycles, an infinite cycle ${(\cdots\,a_1\,\cdots\,a_k\,\,\,a_1\pm n\,\cdots)}$.
If $a$ is in an infinite cycle, we write it as $(\cdots\,a_1\,\,\,a_2\,\cdots\,a_k\,\,\,a_1+qn\,\cdots)$ with $a=a_1$.
Then $\tau\pi$ has a cycle with a sequence $a_1\,\,\,a_2\,\cdots\,a_k\,\,\,a_1+(q\pm1)n$.
This is a new class of finite cycles if and only if $q=\mp1$, and otherwise it is a class of infinite cycles.
\end{proof}

We now assemble some additional fragments of Theorem~\ref{isom}.

\begin{prop}\label{fragments}\quad
\begin{enumerate}[\quad\rm\bf1.]
\item\label{into}
The map $\perm$ maps $\tNCAc$ into $[1,c]_{T\cup L}$.
\item \label{length perm P}
If $\P\in\tNCAc$, then $\ell_{T\cup L}(\perm(\P))$ equals the rank of $\P$ in $\tNCAc$, which is $n$ minus the number of non-annular blocks of $\P$ and which also equals $n$ minus the number of mod-$n$ classes of finite cycles in $\perm(\P)$.
\item \label{perm cov}
If $\P\covered\Q$ in $\tNCAc$, then $\perm(\P)\covered_{T\cup L}\perm(\Q)$.
\end{enumerate}
\end{prop}
\begin{proof}
To begin, we suppose $\P \covered \Q$ and construct a reflection or loop $\tau$ such that $\perm(\Q)=\tau \cdot\perm(\P)$.

Proposition~\ref{conn edge cov} says that there exists a simple connector $\kappa$ for $\P$ such that $\Q=\P\cup\kappa$.
The curve $\kappa$ connects two blocks $E$ and $E'$ of $\P$, possibly with $E'=E$.
In particular, $\kappa$ leaves $E$ by crossing some part of the boundary of $E$ and then enters $E'$ by crossing some part of the boundary of $E'$.  
There are two possibilities:  
Either $\kappa$ leaves $E$ through a segment of the boundary that connects two numbered points and enters $E'$ in the same way, or (up to swapping $E$ and $E'$) $\kappa$ leaves $E$ between two numbered points and enters $E'$ through a circular component of the boundary of $E'$ that is disjoint from the boundary of $A$.
(The latter can happen only when $E'$ is a dangling annular block.)


First, assume that $\kappa$ leaves $E$ between numbered points and enters $E'$ between numbered points.
Let $p$ and $q$ be the two numbered points that $\kappa$ passes between on its way out of $E$, with $p$ on the left as $\kappa$ leaves $E$ and $q$ on the right.
Similarly, let $p'$ be on the right of $\kappa$ as it enters $E'$, and let $q'$ be on the left (so that $p'$ is on the left as $\kappa$ leaves $E'$ in the opposite direction).

Consider the curve in $\Q$ that starts at $q$, follows the boundary of $E$ to $\kappa$, follows~$\kappa$ to the boundary of $E'$, and then follows the boundary of $E'$ to $q'$.  
This curve determines an element $\tau$ that is a reflection (if $q\neq q'$) or a loop (if $q=q'$), as follows:
Let $\T$ be the noncrossing partition whose only nontrivial block is the curve (if $q\neq q'$) or the unique dangling annular block having the curve as a component of its boundary, and set $\tau=\perm(\T)$.
Then $\perm(\Q)=\tau\cdot\perm(\P)$.  

On the other hand, suppose $E'$ is a dangling annular block and $\kappa$ leaves $E$ between numbered points ($p$ on the right and $q$ on the left) and enters $E'$ through the circular component of the boundary of $E'$ that is disjoint from the boundary of $A$.
Let $\tau$ be the loop $\ell_p$ if $E$ is connected to the outer boundary of $A$ or the loop $\ell^{-1}_p$ if $E$ is connected to the inner boundary.
Then multiplying by $\tau$ on the left turns the class of finite cycles associated to $E$ to an infinite cycle.
This infinite cycle is encoded by a component of the boundary of a new non-dangling annular block in $\Q$ containing $E$ and $E'$. 
We conclude that $\perm(\Q)=\tau\cdot\perm(\P)$.

We see that in every case, there exists $\tau\in T\cup L$ such that $\perm(\Q)=\tau\cdot\perm(\P)$.
A maximal chain in $\tNCAc$ has the form $\P_0\covered\cdots\covered\P_n$, where $\P_0$ consists of trivial blocks and $\P_n$ has a single block (the entire annulus $A$).
Since $\perm(\P_0)$ is the identity and $\perm(\P_n)=c$, we can find $\tau_i\in T\cup L$ for each $\P_{i-1}\covered\P_i$ as in the paragraphs above and write a word $\tau_1\cdots\tau_n$ for $c$.
But every word for $c$ in the alphabet $T\cup L$ has at least $n$ letters:
This is because $c$ has no finite cycles, because the identity has $n$ classes of finite cycles, and because Lemma~\ref{only one} says that multiplying on the left by $\tau\in T\cup L$ can increase the number of classes of finite cycles by at most $1$. 
Thus $\tau_1\cdots\tau_n$ is a reduced word for $c$ in the alphabet $T\cup L$.
Since every $\P\in\tNCAc$ is on some maximal chain, we have proved Assertion~\ref{length perm P}.

Now, returning to the earlier assumption that $\P\covered\Q$, since we showed that $\perm(\Q)=\tau\cdot\perm(\P)$, and by Assertion~\ref{length perm P}, we conclude that $\perm(\P)\covered\perm(\Q)$.
This is Assertion~\ref{perm cov}.

Finally, since the unique maximal element of $\tNCAc$ maps to $c$, Assertion~\ref{perm cov} and an easy inductive argument complete the proof of Assertion~\ref{into}.
\end{proof}

For a moment, it will be convenient to restrict our attention to \newword{annular permutations}.
These are permutations with at most 2 infinite cycles, with each infinite cycle either monotone increasing or monotone decreasing.
Furthermore, if there are two infinite cycles, then one is increasing and the other is decreasing.
Note that in an annular permutation, an increasing cycle must be $(\cdots\,a_1\,\cdots\,a_k\,\,\,a_1+n\,\cdots)$ for some $a_1,\ldots,a_k$, because if $a_1+n$ is not in the cycle, then the permutation has another increasing infinite cycle containing $a_1+n$. 
The analogous fact is true for a decreasing cycle.

\begin{lemma}\label{transp act}
Suppose $\pi$ is an annular permutation in $\SZmodn$ and suppose $\tau$ is a reflection or loop.
Then $\tau\pi$ has more mod-$n$ classes of finite cycles than $\pi$ has if and only if $\tau$ is one of the following:
\begin{itemize}
\item\label{same fin}
$\tau=(a\,\,\,b)_n$ with $a$ and $b$ in the same finite cycle of $\pi$.
\item\label{same inf}
$\tau=(a\,\,\,b)_n$ with $a$ and $b$ in the same infinite cycle of $\pi$ and $|a-b|<n$.
\item\label{diff inf}
$\tau=(a\,\,\,b)_n$ with $a$ and $b$ in different infinite cycles of $\pi$.
\item
$\tau=(\cdots\,a\,\,\,a+n\,\cdots)$ and $a$ is in a decreasing infinite cycle of $\pi$.
\item
$\tau=(\cdots\,a\,\,\,a-n\,\cdots)$ and $a$ is in an increasing infinite cycle of $\pi$.
\end{itemize}
When $\tau\pi$ has more mod-$n$ classes of finite cycles than $\pi$ has, then $\tau\pi$ is an annular permutation and has exactly one more class of finite cycles than $\pi$ has.
\end{lemma}
\begin{proof}
We follow the proof of Lemma~\ref{only one} through all of the possibilities for $\tau$, with the additional assumption that $\pi$ is an annular permutation.
We will see that when one of the five conditions in the lemma holds, then $\tau\pi$ is an annular permutation with exactly one more class of finite cycles than $\pi$.
We will also see that, in all other cases, $\tau\pi$ does not have more classes of finite cycles than $\pi$.

If $\tau=(a\,\,\,b)_n$ with $a$ and $b$ in the same finite cycle of $\pi$, then $\tau\pi$ has exactly one more mod-$n$ class of finite cycles.
Since no infinite cycles are created when $\tau$ is applied, $\tau\pi$ is an annular permutation.

If $\tau=(a\,\,\,b)_n$ with $a$ and $b$ in the same infinite cycle of $\pi$, then $\tau\pi$ either has the same number of classes of finite cycles as~$\pi$ or one more.
Since $\pi$ is annular, the infinite cycle is either increasing or decreasing and the value $q$ in the proof of Lemma~\ref{only one} is accordingly either $+1$ or $-1$.
The condition given there for $\tau\pi$ to have an additional class of finite cycles thus becomes the condition that $|a-b|<n$.
When that condition holds, the new infinite cycle in $\tau\pi$ is monotone in the same direction as the infinite cycle in $\pi$ that it was formed from, so $\tau\pi$ is still an annular permutation.

If $\tau=(a\,\,\,b)_n$ with $a$ and $b$ in different finite cycles of $\pi$, then $\tau\pi$ has one fewer class of finite cycles than $\pi$.

If $\tau=(a\,\,\,b)_n$ with $a$ and $b$ in different infinite cycles of $\pi$, then since $\pi$ is annular, the two infinite cycles are in different classes, and furthermore the quantities $q$ and $r$ in the proof of Lemma~\ref{only one} have $|q|=|r|=1$ and $q=-r$.
Thus, two infinite cycles of $\pi$ are replaced by a single class of finite cycles in $\tau\pi$.
Since $\tau\pi$ then has no infinite cycles, it is annular.

If $\tau=(a\,\,\,b)_n$ with $a$ in a finite cycle of $\pi$ and $b$ in an infinite cycle of $\pi$, then $\tau\pi$ has one fewer finite cycle than $\pi$.

Suppose $\tau=(\cdots\,a\,\,\,a\pm n\,\cdots)$.
If $a$ is in a finite cycle of $\pi$, then $\tau\pi$ has one less class of finite cycles than $\pi$.
If $a$ is in an infinite cycle, then $\tau\pi$ has either one more class of finite cycles than $\pi$ or the same number.
Since $\pi$ is an annular permutation, the quantity $q$ in the proof of Lemma~\ref{only one} is either $+1$ or $-1$.
The condition in the proof of Lemma~\ref{only one} for $\tau\pi$ to have an additional class of finite cycles is that either $\tau=(\cdots\,a\,\,\,a+n\,\cdots)$ and $a$ is in a decreasing infinite cycle of $\pi$ or $\tau=(\cdots\,a\,\,\,a-n\,\cdots)$ and $a$ is in an increasing infinite cycle of $\pi$.
\end{proof}

The following proposition is the last piece needed to prove the first assertion of Theorem~\ref{isom}.

\begin{prop}\label{perm inv cov}
Suppose $\sigma\covered\pi$ in $[1,c]_{T\cup L}$ and $\pi=\perm(\Q)$ for some $\Q\in\tNCAc$.
Then there exists $\P\in\tNCAc$ such that $\sigma=\perm(\P)$ and $\P\covered\Q$.
\end{prop}
\begin{proof}
Since $\sigma\covered\pi$ in $[1,c]_{T\cup L}$, we can write $\sigma=\tau\pi$ for some $\tau\in T\cup L$.
Furthermore, $\ell_{T\cup L}(\sigma)=\ell_{T\cup L}(\pi)-1$.
Proposition~\ref{fragments}.\ref{length perm P} says that $\ell_{T\cup L}(\pi)$ is $n$ minus the number of non-annular blocks of $\Q$ and also $n$ minus the number of mod-$n$ classes of finite cycles in $\pi$.
We conclude that $\ell_{T\cup L}(\sigma)$ is $n$ minus the number of mod-$n$ classes of finite cycles in $\sigma$, by the same argument as in the proof of Proposition~\ref{fragments}.
(By Lemma~\ref{only one}, multiplying on the left by $\tau\in T\cup L$ can increase the number of classes of finite cycles by at most $1$, and the identity element has $n$ classes of finite cycles.)
Thus $\sigma$ has one more class of finite cycles than $\pi$, so~$\tau$ is described by one of the bullet points in Lemma~\ref{transp act}.
Looking at the proof of Lemma~\ref{only one} in each case, we now find a cutting curve $\lambda$ for $\Q$ (in the sense of Proposition~\ref{cut curve cov}) that creates a noncrossing partition $\P$ with $\perm(\P)=\sigma$.

If $\tau=(a_1 \,\,\, a_i)_n$ for some class $(a_1 \,\,\, a_2\,\cdots\,a_k)_n$ of $\pi$ and $1<i\le k$, then $\lambda$ cuts the corresponding disk in $\Q$, beginning between $a_k\mod n$ and $a_1\mod n$ and ending between $a_{i-1}\mod n$ and $a_i\mod n$.

If $\tau=(a_1 \,\,\, a_i)_n$ for some increasing cycle $(\cdots\,a_1\,\cdots\,a_k\,\,\,a_1+n\,\cdots)$ of $\pi$ and ${1<i\le k}$, then $\lambda$ cuts the corresponding annular block in $\Q$, beginning between $a_k\mod n$ and $a_1\mod n$, cutting off a disk with vertices $a_1,\ldots,a_{i-1}$, and ending between $a_{i-1}\mod n$ and $a_i\mod n$.
If $\tau=(a\,\,\,b)_n$ for $a$ and $b$ in the same decreasing infinite cycle of $\pi$ and $|a-b|<n$, then the construction of $\lambda$ is similar.

Suppose $\tau=(a_1\,\,\,b_1)$ for an increasing cycle $(\cdots\,a_1\,\cdots\,a_k\,\,\,a_1+n\,\cdots)$ of $\pi$ and a decreasing cycle $(\cdots\,b_1\,\cdots\,b_m\,\,\,b_1-n\,\cdots)$.
Then $\lambda$ cuts the corresponding annular block into a disk, starting between $a_k\mod n$ and $a_1\mod n$ and ending between $b_m\mod n$ and $b_1\mod n$, but wrapping around the annulus so that the curve from $a_1\mod n$ to $b_1\mod n$ that follows $\lambda$ closely corresponds to $\tau$.

If $\tau=(\cdots\,a_1\,\,\,a_1\pm n\,\cdots)$ and $(\cdots\,a_1\,\,\,a_2\,\cdots\,a_k\,\,\,a_1 \mp n\,\cdots)$ is an increasing or decreasing cycle of $\pi$, then there are two possibilities.
If the corresponding annular block $E$ is dangling, then $\lambda$ is a curve starting between $a_k\mod n$ and $a_1\mod n$ and ending on the component of $E$ with no numbered points.
If $E$ is non-dangling, then $\lambda$ starts between $a_k\mod n$ and $a_1\mod n$, goes around the annulus once, and ends between $a_k\mod n$ and $a_1\mod n$.
\end{proof}

We now prove the first assertion of Theorem~\ref{isom}.

\begin{prop}\label{isom TL}
The map $\perm:\tNCAc\rightarrow \SZmodn$ is an isomorphism from $\tNCAc$ to the interval $[1,c]_{T\cup L}$ in $\SZmodn$.   
\end{prop}
\begin{proof}
Propositions~\ref{one to one} and~\ref{fragments} show that $\perm$ is a one-to-one, order-preserving map from $\tNCAc$ to $[1,c]_{T\cup L}$.
Since $c$ is the image, under $\perm$, of the maximal element of $\tNCAc$, Proposition~\ref{perm inv cov} and an easy inductive argument show that $\perm$ is onto $[1,c]_{T\cup L}$.
Then Proposition~\ref{perm inv cov} also implies that the inverse map to $\perm$ is order-preserving.
\end{proof}

To prove the second assertion of Theorem~\ref{isom}, we prove analogous facts about the restriction of $\perm$ to $\tNCAcircc$, by the analogous proof, reusing some arguments.

\begin{prop}\label{fragments T}\quad
\begin{enumerate}[\quad\rm\bf1.]
\item\label{into T}
The restriction of $\perm$ maps $\tNCAcircc$ into $[1,c]_T$.
\item \label{length perm P T}
If $\P\in\tNCAcircc$, then $\ell_T(\perm(\P))$ equals the rank of $\P$ in $\tNCAcircc$, which is $n$ minus the number of non-annular blocks of $\P$ and which also equals $n$ minus the number of mod-$n$ classes of finite cycles in $\perm(\P)$.
\item \label{perm cov T}
If $\P\covered\Q$ in $\tNCAcircc$, then $\perm(\P)\covered_T\perm(\Q)$.
\end{enumerate}
\end{prop}
\begin{proof}
Suppose $\P\covered\Q$ in $\tNCAcircc$.
Then Proposition~\ref{cov circ} says that $\P\covered\Q$ in $\tNCAc$, so we can follow the same argument as in the proof of Proposition~\ref{fragments}.
Since $\P\in\tNCAcircc$, we ignore the case where the block $E'$ is a dangling annular block.
In the other case, recall that we constructed a loop $\tau$ if and only if $q=q'$.
However, Proposition~\ref{cov circ} rules out that possibility.  
Thus we construct a reflection $\tau$ such that $\perm(\Q)=\tau\cdot\perm(\P)$.
We construct a reduced word for $c$ in the alphabet $T$ and complete the proof just as in the proof of Proposition~\ref{fragments}.
\end{proof}

\begin{prop}\label{perm inv cov T}
Suppose $\sigma\covered\pi$ in $[1,c]_T$ and $\pi=\perm(\Q)$ for some $\Q\in\tNCAcircc$.
Then there exists $\P\in\tNCAcircc$ such that $\sigma=\perm(\P)$ and $\P\covered\Q$.
\end{prop}
\begin{proof}
Since $\sigma\covered\pi$ in $[1,c]_T$, we can write $\sigma=\tau\pi$ for some $\tau\in T$.
As in the proof of Proposition~\ref{perm inv cov}, but using Proposition~\ref{fragments T} instead of Proposition~\ref{fragments}, we see that $\ell_T(\sigma)$ is $n$ minus the number of mod-$n$ classes of finite cycles in $\pi$.
The reflection $\tau$ is described by one of the first three bullet points in Lemma~\ref{transp act}.
We construct a cutting curve as in the proof of Proposition~\ref{perm inv cov}, using only the cases where $\tau$ is a reflection.
In those cases, cutting $\Q$ to make $\P$ does not create a dangling annular block, so $\P\in\tNCAcircc$.
Proposition~\ref{cov circ} says that $\P\covered\Q$ in $\tNCAcircc$ as well.
\end{proof}

The following proposition completes the proof of Theorem~\ref{isom}.
\begin{prop}\label{isom T}
The map $\perm$ restricts to an isomorphism from $\tNCAcircc$ to the interval $[1,c]_T$ in $\Stilde_n$.
\end{prop}
\begin{proof}
We argue as in the proof of Proposition~\ref{isom TL}, using Propositions~\ref{fragments T} and~\ref{perm inv cov T} in place of Propositions~\ref{fragments} and~\ref{perm inv cov}.
\end{proof}

\subsection{Complements and intervals}\label{comp int sec}
In this section, we define the Kreweras complement on $\tNCAc$ on $\tNCAcircc$ and discuss product decompositions of intervals.

The \newword{Kreweras complement} is an anti-automorphism $\Krew:\tNCAc\to\tNCAc$, which restricts to an anti-automorphism $\tNCAcircc\to\tNCAcircc$.
This name refers to the original complementation map on noncrossing partitions in~\cite{Kreweras}.
Analogous constructions exist in types B and~D.
(See~\cite{Ath-Rei,Reiner}.)

For each numbered point $i$, choose a point $i'$ on the boundary of the annulus~$A$ between $i$ and $c(i)$.
(Recall from Lemma~\ref{c annulus} that $c$ maps an outer point $i$ to the next outer point clockwise from $i$ and maps an inner point $i$ to the next inner point counterclockwise.)
Given $\P\in\tNCAc$, first construct a noncrossing partition~$\P'$ on the annulus with numbered points $1',\ldots,n'$ whose blocks are the maximal embedded blocks that are disjoint from the nontrivial blocks of $\P$, together with trivial blocks $\set{i'}$ for every $i'$ contained in a block of $\P$.
Then $\Krew(\P)$ is obtained from~$\P'$ by applying a homeomorphism of $A$ that rotates each~$i'$ back to $i$ while moving interior points of $A$ as little as possible.
(That is, the homemorphism maps a radial segment from the inner boundary to the outer boundary to a segment that does not wrap around the annulus.)

\begin{example}\label{Krew ex}
Figure~\ref{Krew exs} shows the Kreweras complements of the noncrossing partitions shown in Figure~\ref{nc exs}.
\end{example}

\begin{figure}
\begin{tabular}{cccc}
\includegraphics{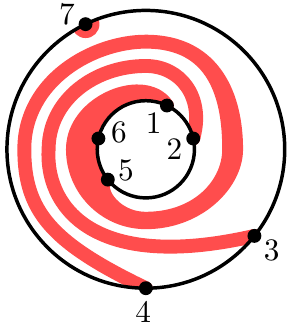}
&\includegraphics{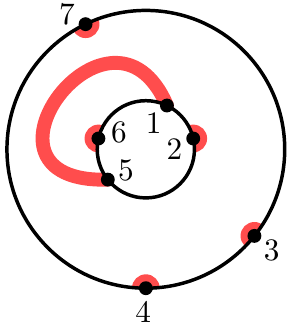}
&\includegraphics{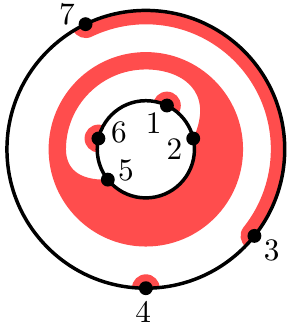}
\end{tabular}
\caption{Kreweras complements of the noncrossing partitions from Figure~\ref{nc exs}}
\label{Krew exs}
\end{figure}

The map $w\mapsto w^{-1}c$ is an anti-automorphism of $[1,c]_T$ also often called the \newword{Kreweras complement}.  
The same map on the larger interval $[1,c]_{T\cup L}$ is an anti-automorphism for the same easy reasons.

\begin{theorem}\label{Krew thm}
If $\P\in\tNCAc$, then $\perm(\Krew(\P))=\perm(\P)^{-1}c$.
\end{theorem}

\begin{proof}
Since the map $\Krew$ is an anti-automorphism, by Theorem~\ref{join curve} it is completely determined by its action on (the noncrossing partitions associated to) arcs and boundary segments.
Thus since $w\mapsto w^{-1}c$ is also an anti-automorphism, it is enough to show that $\perm(\Krew(\P))=\perm(\P)^{-1}c$ when $\P$ is a noncrossing partition associated to an arc or boundary segment.

Write 
\[c=(\cdots\,a_1\,\,\,a_2\,\cdots\,a_k\,\,\,a_1+n\,\cdots)(\cdots\,b_1\,\,\,b_2\,\cdots\,b_{n-k}\,\,\,b_1-n\,\cdots)\]
as in Lemma~\ref{c annulus}.
Taking $\pi=c$ and $\tau=\perm(\P)^{-1}$ in Lemma~\ref{transp act}, we see that $\tau$ satisfies one of the five bullet points (not the first, because $c$ has no finite cycles).
In each of these cases, looking back at the proof of Lemma~\ref{transp act}, we easily verify that $\tau\pi$ is $\perm(\Krew(\P))$.
We omit the details.
\end{proof}

By \cite[Proposition~2.13]{surfnc}, the lower interval below some $\P\in\tNCAc$ is isomorphic to a product of noncrossing partition lattices, one noncrossing partition lattice for every block $E$ of $\P$.
The factor for $E$ is a lattice of ``noncrossing partitions of a marked surface'' in the sense of~\cite{surfnc}, where the ``marked surface'' is the block $E$ itself, ``marked'' with the numbered points it contains.  
Thus disk blocks (including degenerate disks) in $\P$ give rise to factors that are finite type\nobreakdash-A noncrossing partition lattices.
Non-dangling annular blocks give rise to lattices of noncrossing partitions of the smaller annuli.
Perhaps surprisingly, dangling annular blocks give rise to lattices of noncrossing partitions of finite type~B.
(This is an observation due to Laura Brestensky in connection with~\cite{BThesis} and appears as \cite[Theorem~4.2]{surfnc}.
It was also observed in \cite{ALPU}.)

By Theorem~\ref{isom}, we can apply these considerations to intervals in $\Stilde_n$ and $\SZmodn$.
Furthermore, since all the intervals in question are self-dual (as discussed above for affine type $\afftype{A}$ and as is well known in finite type), the results for lower intervals apply more broadly to arbitrary intervals.
We summarize in the following proposition.

\begin{prop}\label{int prop}
Each interval in $[1,c]_{T\cup L}$ is isomorphic to a product, each of whose factors is of one of the following forms:
\begin{itemize}
\item $[1,c']_{T'\cup L'}$ for $c'$, $T'$, and $L'$ associated to a Coxeter group of affine type~\textup{$\afftype{A}$}.
\item $[1,c']_{T'}$ for $c'$ and $T'$ in a Coxeter group of finite type \textup{A} or \textup{B}.
\end{itemize}
\end{prop}

Extending the argument for \cite[Proposition~2.13]{surfnc} in a straightforward way, one can see that lower intervals in $\tNCAcircc$ are products in the analogous way.
We recover the following proposition.

\begin{prop}\label{int prop circ}
Each interval in $[1,c]_{T\cup L}$ is isomorphic to a product, each of whose factors is of the form $[1,c']_{T'}$ for $c'$ and $T'$ in a Coxeter group of affine type \textup{$\afftype{A}$} or finite type \textup{A}.
\end{prop}

\section{Affine type C}\label{aff type c}
In this section, we use our results about affine type $\afftype{A}$ to produce the analogous results about affine type $\afftype{C}$, by the folding technique described in Section~\ref{folding sec}.
The same results can be obtained directly from the ``projecting to the Coxeter plane'' construction, but folding serves as a shortcut.
Our treatment of affine type C owes a debt to \cite[Section~8.4]{Bj-Br}, but we adopt different conventions.

\subsection{Symmetric noncrossing partitions of an annulus}\label{sym ann sec}
We realize the Coxeter group of type $\afftype{C}_{n-1}$ as the group $\Stildes$ of \newword{affine signed permutations}.   
These are the permutations $\pi:\integers\to\integers$ with $\pi(i+2n)=\pi(i)+2n$ for all $i\in\integers$ and also $\pi(-i)=-\pi(i)$ for all $i\in\integers$.
In particular, $\pi(0)=0$ and also $\pi(n)=-\pi(-n)=-(\pi(n)-2n)$, so $\pi(n)=n$.
Thus $\pi$ fixes all multiples of $n$.


The conditions $\pi(i+2n)=\pi(i)+2n$ and $\pi(-i)=-\pi(i)$ imply that $\sum_{i=1}^{2n}\pi(i)=\binom{2n+1}2$, so $\Stildes$ is a subgroup of the group $\Stilde_{2n}$ of affine permutations, a Coxeter group of type $\afftype{A}_{2n-1}$.
However, it is more convenient to consider it as a subgroup of a different Coxeter group.
Let $W'$ be the subgroup of $\Stilde_{2n}$ consisting of elements that fix all multiples of $n$.
We continue the subscript $_{2n}$ notation for mod-$2n$ families of cycles and define elements $s'_0,\ldots,s'_{2n-3}$ of $W'$ as 
\begin{align*}
s'_0&=(-1\,\,\,1)_{2n},\\
s'_i&=(i\,\,\,i+1)_{2n}\text{ for }i=1,\ldots,n-2,\\
s'_{n-1}&=(n-1\,\,\,n+1)_{2n},\text{ and}\\
s'_i&=(i+1\,\,\,i+2)_{2n}\text{ for }i=n,\ldots,2n-3.
\end{align*}
By the obvious isomorphism with $\Stilde_{2n-2}$ (ignoring all multiples of $n$), we see that~$W'$ is a Coxeter group of type $\afftype{A}_{2n-3}$, with simple reflections $s'_0,\ldots,s'_{2n-3}$.

The Coxeter group $\Stildes$ of type $\afftype{C}_{n-1}$ is the subgroup of $W'$ whose elements satisfy $\pi(-i)=-\pi(i)$ for all $i\in\integers$.
To write cycles in $\Stildes$, we imitate~\cite{Ath-Rei} by introducing notation $(\!(a_1\,\cdots\,a_k)\!)_{2n}$ for the family of cycles consisting of cycles $(a_1\,\cdots\,a_k)_{2n}$ and $(-a_1\,\cdots\,-a_k)_{2n}$.
Similarly, $(\!(\cdots\,a_1\,\,\,a_2\,\cdots\,a_\ell\,\,\,a_1+2qn\,\cdots)\!)$ means 
\[(\cdots\,a_1\,\,\,a_2\,\cdots\,a_\ell\,\,\,a_1+2qn\,\cdots)(\cdots\,-a_1\,\,\,-a_2\,\cdots\,-a_\ell\,\,\,-a_1-2qn\,\cdots).\]
In this notation, the simple reflections in $\Stildes$ are 
\begin{align*}
s_0=s'_0&=(-1\,\,\,1)_{2n},\\
s_i=s'_is'_{2n-2-i}&=(\!(i\,\,\,i+1)\!)_{2n}\text{ for }i=1,\ldots,n-2,\text{ and}\\
s_{n-1}=s'_{n-1}&=(n-1\,\,\,n+1)_{2n}.
\end{align*}
The reflections in $\Stildes$ are of two types:
For each $i,j\in\integers\setminus\set{\ldots,-n,0,n,\ldots}$ with neither $j$ nor $-j$ equivalent to $i\mod {2n}$, there is a reflection $(\!(i\,\,\,j)\!)_{2n}$.
For each $i,j\in\integers\setminus\set{\ldots,-n,0,n,\ldots}$ with $-i$ equivalent to $j\mod {2n}$, there is a reflection $(i\,\,\,j)_{2n}$.
A given reflection is named by many different choices of $i$ and~$j$.

Let $\phi$ be the automorphism of $W'$ sending a permutation $\pi$ in $W'$ to a permutation $\phi(\pi)$ given by $\phi(\pi)(i)=-\pi(-i)$.
Thus $\Stildes$ is the subgroup of $W'$ consisting of permutations fixed by $\phi$.
We see that $\Stildes$ is a folding of $W'$ in the sense of Section~\ref{folding sec}.
In particular, each Coxeter element $c$ of $\Stildes$ is also a Coxeter element of~$W'$, because each simple reflection of $\Stildes$ is a product of one or two simple reflections of~$W'$.

We describe the choice of Coxeter element for $\Stildes$ in terms of choosing a sign on each of the integers $1,\ldots,n-1$.
The Coxeter diagram for $\Stildes$ is a path with vertices $s_0,\ldots,s_{n-1}$, so the choice of a Coxeter element $c$ amounts to choosing, for each $i=1,\ldots,n-1$, whether $s_{i-1}$ is before $s_i$ or after $s_i$ is every reduced word for $c$.
We record this choice by choosing a positive sign on $i$ if $s_{i-1}\to s_i$ or choosing a negative sign on $i$ if $s_i\to s_{i-1}$.
This choice determines a subset of $\set{\pm1,\ldots,\pm (n-1)}$ consisting of $n-1$ elements with pairwise distinct absolute values.
We will refer to this subset as a \newword{signing} of $\set{1,\ldots,n-1}$.

Comparing the conventions of inner and outer points in affine type A with the convention here about signings of $\set{1,\ldots,n-1}$, we see that each $i\in\set{1,\ldots,n-1}$ is in the signing if and only if $i$ is outer and $2n-i$ is inner, and $-i$ is in the signing if and only if $i$ is inner and $2n-i$ is outer.
(Recall that, in this construction of the group $\Stildes$ of type $\afftype{C}_{n-1}$ we omit multiples of $n$ entirely, so neither $n$ nor $2n$ appears among the inner and outer elements.)

Applying Lemma~\ref{c annulus}, we can write $c$ in cycle notation, but it is more convenient to write our choice of inner and outer elements differently first.
Rather than listing the elements of $\set{1,\ldots,2n}\cup\set{n+1,\ldots,2n-1}$, increasing in the clockwise direction, on the inner and outer boundaries of the annulus, we instead write the elements of $\set{\pm1,\ldots,\pm(n-1)}$, increasing in the clockwise direction, on the boundaries of the annulus.
More specifically, we place the numbers $-(n-1),\ldots,-1$ on the left of the annulus and the numbers $1,\ldots,n-1$ on the right, with the clockwise angle from the bottom of the annulus to a negative number $-i$ equal to the counterclockwise angle from the bottom to the positive number $i$.
If $i\in\set{1,\ldots,n-1}$ is in the signing, then $i$ appears on the outer boundary and $-i$ appears on the inner boundary.
If $i\in\set{1,\ldots,n-1}$ is has $-i$ in the signing, then $-i$ appears on the outer boundary and $i$ appears on the inner boundary.  
Thus the signing of $\set{\pm1,\ldots,\pm(n-1)}$ records the outer points, so Lemma~\ref{c annulus} becomes the following lemma:

\begin{lemma}\label{c annulus C}
Let $c$ be a Coxeter element of $\Stildes$, represented as a signing of $\set{1,\ldots,n-1}$.
If $a_1,\ldots,a_{n-1}$ are the elements of the signing in increasing order, then $c=(\!(\cdots\,a_1\,\,\,a_2\,\cdots\,a_{n-1}\,\,\,a_1+2n\,\cdots)\!)$.
\end{lemma}

\begin{example}\label{C ann ex}
When $n=7$ and $c=s_6s_4s_3s_0s_1s_2s_5$, the corresponding signing is shown in Figure~\ref{C ann fig} as a placement of the numbers $\pm1,\ldots,\pm6$ on the boundary of the annulus.
(The dotted circles and gray segment are explained below.)
The cycle notation for~$c$~is $(\!(\cdots\,-6\,\,-4\,\,-3\,\,\,1\,\,\,2\,\,\,5\,\,\,8\,\cdots)\!)$.
\end{example}

\begin{figure}
\includegraphics{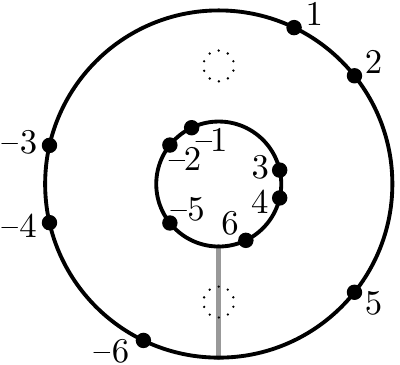}
\caption{The symmetric annulus with numbered points corresponding to $c=s_6s_4s_3s_0s_1s_2s_5$.}
\label{C ann fig}
\end{figure}

Given a Coxeter element $c$ of $\Stildes$, we write $[1,c]_T^C$ for the interval between~$1$ and~$c$ in the absolute order on the group $\Stildes$ and write $[1,c]_T^A$ for the interval between~$1$ and~$c$ in the absolute order on the group $W'$ of type $\afftype{A}_{2n-3}$ described above.
Proposition~\ref{fold subposet} says that $[1,c]_T^C$ is the subposet of $[1,c]_T^A$ induced by elements $w\in[1,c]_T^A$ such that $\phi(w)=w$.

In Section~\ref{nc ann sec}, we defined a map $\perm$ from $\tNCAc$ to $\SZmodn$.
The same map can be defined from the set of noncrossing partitions of the annulus with numbered points $\pm1,\ldots,\pm n$ to the group $S_\integers\cmod{2n}$.  
However, since the numbered points are placed in increasing order clockwise starting from the bottom of the annulus, we take the \emph{date line} to be the vertical radial segment at the bottom of the annulus, marked in gray in Figure~\ref{C ann fig}.
We reuse the name $\perm$ for this map.

Because of the placement of the numbered points $\pm1,\ldots,\pm(n-1)$ on the annulus, the involution $\phi$ on $[1,c]_T^A$ corresponds (via the map $\perm$) to a symmetry of the annulus.
For the moment, write $\phi_g$ (suggesting a ``geometric'' version of $\phi$) for this symmetry, which was described in the introduction by viewing the annulus as a cylinder.
Alternatively, viewing the annulus in polar coordinates $(r,\theta)$ as the region defined by $1\le r\le2$, the map $\phi_g$ sends $(r,\theta)$ to $(2-r,\pi-\theta)$.
In Figure~\ref{C ann fig}, the two fixed points of $\phi_g$ are marked with small dotted circles.
Is is apparent that ${\phi=\perm\circ\phi_g\circ\perm^{-1}}$, and thus we abuse notation by writing simply $\phi$ in place of~$\phi_g$.
Write $\tNCAphic$ for the subposet of $\tNCAc$ induced by \newword{symmetric noncrossing partitions of the annulus} (noncrossing partitions that are fixed by $\phi$).
By the same simple argument that proves Corollary~\ref{fold sublattice} from Proposition~\ref{fold subposet}, $\tNCAphic$ is a sublattice of $\tNCAc$.
The fact that $\tNCAphic$ is a lattice also appears as \cite[Theorem~4.4]{surfnc}, which also shows that $\tNCAphic$ is graded and gives the rank function, as a special case of \cite[Theorem~3.18]{surfnc}:

\begin{theorem}\label{C tilde main}
$\tNCAphic$ is a graded lattice, with rank function given by $n-1$ minus the number of symmetric pairs of distinct disk blocks plus the number of annular~blocks.
\end{theorem} 

We also quote from~\cite{surfnc} a characterization of cover relations in $\tNCAphic$.
Given a symmetric noncrossing partition~$\P$ of the annulus, a \newword{simple symmetric connector} for $\P$ is a symmetric arc $\alpha$ that is a simple connector for $\P$.
A \newword{simple symmetric pair of connectors} is symmetric pair $\alpha,\phi(\alpha)$ of disjoint arcs, each of which is a simple connector for $\P$, but ruling out the possibility that $\alpha$ and $\phi(\alpha)$ combine with blocks of $\P$ to bound an annulus.
The \newword{augmentation of $\P$} along a symmetric simple connector $\alpha$ is defined exactly as before, but with all ``thickenings'' of curves chosen symmetrically.
The \newword{augmentation of $\P$} along a simple symmetric pair of connectors $\alpha$ and $\phi(\alpha)$ is obtained as the augmentation along~$\alpha$, further augmented along $\phi(\alpha)$, with the thickenings in the curve unions chosen so as to make $\P\cup\alpha\cup\phi(\alpha)$ symmetric, and possibly adjoining an additional disk.
Specifically, when $\alpha$ and $\phi(\alpha)$ combine with blocks of $\P$ to bound a disk (necessarily containing a fixed point of $\phi$), we adjoin that disk.
The augmentation, in either case, is denoted $\P\cup\alpha\cup\phi(\alpha)$, since $\alpha=\phi(\alpha)$ for a symmetric simple connector.


The definition of simple symmetric (pairs of) connectors is a special case of \cite[Definition~3.23]{surfnc}, and the following proposition is a special case of \cite[Proposition~3.27]{surfnc}.

\begin{prop}\label{sym covers}
Two noncrossing partitions $\P,\Q\in\tNCAcircc$ have $\P\covered\Q$ if and only if there exists a simple symmetric connector $\alpha$ or simple symmetric pair of connectors $\alpha,\phi(\alpha)$ for $\P$ such that $\Q=\P\cup\alpha\cup\phi(\alpha)$.
\end{prop}

A noncrossing partition of the annulus that is fixed by $\phi$ cannot have a dangling annular block, because it can only have one annular block.
Thus $\tNCAphic$ is also the subposet of $\tNCAcircc$ induced by noncrossing partitions of the annulus that are fixed by $\phi$.
Since $[1,c]_T^C$ is the subposet of $[1,c]_T^A$ induced by elements fixed by $\phi$, the following theorem is an immediate consequence of Theorem~\ref{isom}.

\begin{theorem}\label{isom annular C}
The map $\perm:\tNCAc\to[1,c]_T^A$ restricts to an isomorphism from $\tNCAphic$ to $[1,c]_T^C$.
\end{theorem}
Combining Theorems~\ref{C tilde main} and~\ref{isom annular C}, we recover the theorem of \cite{Digne2} that the interval $[1,c]_T^C$ is a lattice for any choice of $c$.

\subsection{Noncrossing partitions of a disk with two orbifold points}\label{nc 2 orb sec}
We now use the results of Section~\ref{sym ann sec} to give a model for $[1,c]_T^C$ in the orbifold obtained as the quotient of the annulus modulo $\phi$.

The map $\phi$ on the annulus fixes exactly two points (marked with dotted circles in Figure~\ref{C ann fig}). 
The left picture of Figure~\ref{fold ann fig} also shows a dashed circle that cuts the annulus into two pieces, with the action of $\phi$ mapping each piece to the other.
\begin{figure}
\includegraphics{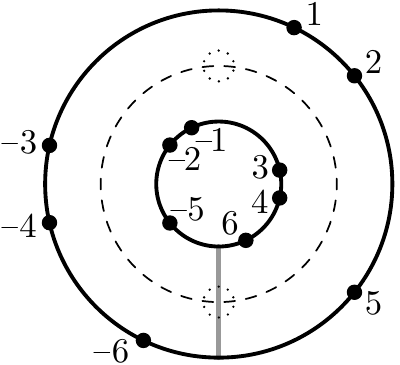}
\qquad
\includegraphics{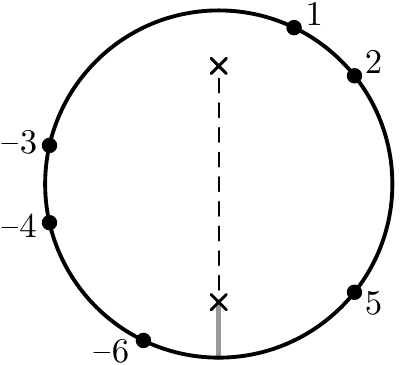}
\caption{The symmetric annulus and the two-orbifold disk}
\label{fold ann fig}
\end{figure}
Thus the quotient of the annulus modulo $\phi$ can be obtained from the annulus by deleting everything inside the dashed circle and identifying each point on the dashed circle with the other point at the same height in the picture.
The result, shown in the right picture of Figure~\ref{fold ann fig}, is a disk with two (order-$2$) orbifold points. 
The image, in the quotient, of the dashed circle, is called the \newword{toggle line}.
The numbered points on the outside of the disk are the points in the signing, in increasing order, exactly as on outside boundary of the annulus. 
We write $C$ for this disk, with these numbered points, and call $C$ the \newword{two-orbifold disk}.

We now define noncrossing partitions of the two-orbifold disk.
For easy comparison, the exposition here parallels Section~\ref{nc ann sec}.
In light of Theorem~\ref{isom annular C}, we want to make the definition such that noncrossing partitions of the two-orbifold disk are precisely the quotients, modulo $\phi$, of the noncrossing partitions of the annulus that are fixed by $\phi$.

Two subsets of $C$ are related by \newword{ambient isotopy} if they are related by a homeomorphism from $C$ to itself, fixing the boundary $\partial C$ pointwise and fixing each orbifold point, and homotopic to the identity by a homotopy that fixes $\partial C$ pointwise and fixes each orbifold point at every step.

A \newword{boundary segment} is the portion of the boundary of $C$ between two adjacent numbered points.
An \newword{arc} in $C$ is a curve in $C$ that does not intersect itself except possibly at its endpoints, does not intersect the boundary or orbifold points of~$C$, except possibly at its endpoints, and is of one of the following two types:
An \newword{ordinary arc} has both endpoints on the boundary of $C$, while an \newword{orbifold arc} has one endpoint on the boundary and the other at an orbifold point.
An arc may not bound a monogon in $C$ unless that monogon contains both orbifold points, and it may not combine with a boundary segment to bound a digon unless that digon contains one or both orbifold points.
Arcs are considered up to ambient isotopy.

An \newword{embedded block} in $C$ is one of the following:
\begin{itemize}
\item
a \newword{trivial block}, meaning a singleton consisting of a numbered point in $C$;
\item
an arc (ordinary or orbifold) or boundary segment in $C$; or
\item
a \newword{disk block}, meaning a closed disk in $C$ whose boundary is a union of ordinary arcs and/or boundary segments of $C$.
\end{itemize}
The first two types of blocks are \newword{degenerate disk blocks}.
An embedded block can contain one or both orbifold points, but a non-degenerate disk block cannot contain an orbifold point on its boundary because of the requirement that its boundary is a union of \emph{ordinary} arcs and/or boundary segments.
Embedded blocks are considered up to ambient isotopy.

A \newword{noncrossing partition} of $C$ is a collection of pairwise disjoint embedded blocks such that every numbered point is contained in one of the embedded blocks.
Noncrossing partitions are considered up to ambient isotopy.
As before, isotopy representatives of a noncrossing partition are called \newword{embeddings} of the noncrossing partition.

\begin{example}\label{some ncs orb}
The noncrossing partitions of the two-orbifold disk on bottom row of Figure~\ref{nc C exs} exemplify the case where $n=7$ and $c=s_6s_4s_3s_0s_1s_2s_5$ (in the Coxeter group $\Stildes[14]$ of type $\afftype{C}_6$).
The top row shows the corresponding symmetric noncrossing partitions on an annulus.
(Compare Figures~\ref{C ann fig} and~\ref{fold ann fig}.) 
Figure~\ref{one more nc C ex} shows one more example.
\end{example}

\begin{figure}
\begin{tabular}{cc}
\scalebox{1.6}{\includegraphics{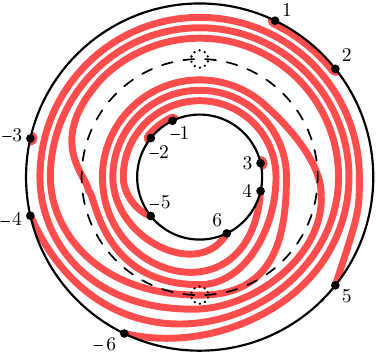}}
&\!\!\!\!\scalebox{1.6}{\includegraphics{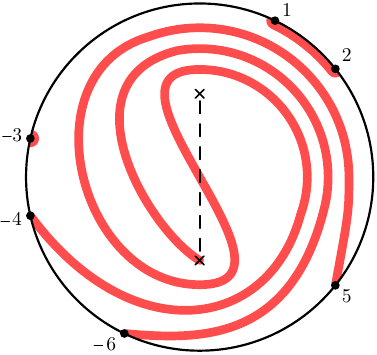}}\\
&$\,\,\,\,\,\P_4$
\end{tabular}
\caption{Another symmetric noncrossing partition of an annulus and its corresponding noncrossing partition of the two-orbifold disk}
\label{one more nc C ex}
\end{figure}

Write $\tNCCc$ for the set of noncrossing partitions of the two-orbifold disk $C$ associated to the Coxeter element $c$, partially ordered with $\P \leq \Q$ if and only if there exist embeddings of $\P$ and $\Q$ such that each block in $\P$ is a subset of a block in~$\Q$.

Let $q:A\to C$ be the quotient map from the annulus to the two-orbifold disk.  
Re-use the symbol $q$ for the map $\tNCAphic\to\tNCCc$ that applies $q$ to every block.
Write~$\imath$ for the map $\tNCCc\to\tNCAphic$ that sends every block of a noncrossing partition to the one or two blocks that constitute its preimage under~$q$.

Some comments are in order on the well-definition of these maps.
First, the maps are defined on isotopy classes.
Well-definition of $\imath$, in this sense, is straightforward because an ambient isotopy on $C$ lifts to an ambient isotopy on $A$ in a straightforward manner.
The map $q$ is also well defined, in this sense, because equivalence classes of noncrossing partitions of $A$ are determined entirely by combinatorial data, and this combinatorial data also determines the image of a noncrossing partition under~$q$.
Second, the maps must in fact map noncrossing partitions to noncrossing partitions.
For both maps, the point is that symmetric annnular blocks in the annulus correspond to disks in $C$ containing both orbifold points, that symmetric disks in the annulus correspond to disks in $C$ containing one orbifold point, and that symmetric pairs of disks in the annulus correspond to disks in $C$ not containing either orbifold point.
For both maps, the ``noncrossing'' part is immediate.

The two maps are inverse to each other, and each preserves containment relationships.
Thus we have the following proposition.

\begin{proposition}\label{C sym same}
The map $\imath:\tNCCc\to\tNCAphic$ is an order isomorphism, with inverse $q$.
\end{proposition}

Thus Theorem~\ref{C tilde main} implies that $\tNCCc$ is a graded lattice with rank function given by $n-1$ minus the number of blocks containing no orbifold points plus the number of blocks containing $2$ orbifold points.
Rephrasing the rank function, we have the following theorem.

\begin{theorem}\label{C tilde orb main}
$\tNCCc$ is a graded lattice, with rank function given by 
\[(n-1)-(\#\text{ blocks of }\P)+(\#\text{ orbifold points enclosed by blocks of }\P).\]
\end{theorem}

We can also characterize covers in $\tNCCc$ by adapting Proposition~\ref{sym covers}.
Given $\P\in\tNCCc$, a \newword{simple connector} for $\P$ is an arc or boundary segment $\alpha$ in $C$ that does not have an isotopy representative contained in a block of $\P$ but starts in some block $E$ of $\P$, leaves $E$, then either ends at an orbifold point without entering another block or enters some block of $E'$ of $\P$, which it does not leave again.
We allow the possibility that $E'=E$, but if so, we impose an additional requirement:
Since $E'=E$, because $\alpha$ has no isotopy representative in $\P$, $\alpha$ combines with $E$ to enclose at least one orbifold point not contained in a block of $\P$.
We disallow the case where neither orbifold point is in a block of $\P$ but $\alpha$ combines with $E$ to enclose both orbifold points.
The \newword{augmentation of $\P$} along a simple connector $\alpha$ is the noncrossing partition $\P\cup\alpha$ obtained from $\P$ by replacing $E$ and $E'$ by the union of $E$ and $E'$ and a thickened $\alpha$, adjoining any digons bounded by this union and boundary segments of $C$ and also, if $E=E'$, adjoining the disk enclosed by $\alpha$ and $E$. 
We restate Proposition~\ref{sym covers} in the two-orbifold disk:

\begin{prop}\label{sym covers C}
Two noncrossing partitions $\P,\Q\in\tNCCc$ have $\P\covered\Q$ if and only if there exists a simple connector $\alpha$ for $\P$ such that $\Q=\P\cup\alpha$.
\end{prop}

Write $\perm^C:\tNCCc\to\afftype{C}_{n-1}$ for the composition $p\circ\imath$.
Combining Theorem~\ref{isom annular C} and Proposition~\ref{C sym same}, we have the following theorem.

\begin{theorem}\label{isom 2 orb}
The map $\perm^C$ is an isomorphism from $\tNCCc$ to $[1,c]_T^C$.
\end{theorem}

We will describe the map $\perm^C$ directly on a noncrossing partition $\P$ of the two-orbifold disk.
We first thicken each degenerate disk block to an actual disk (possibly containing an orbifold point in its interior).
We read a cycle from each block, consisting of the points of the signing that we visit as we trace  along the boundary, but we negate them and/or add a multiple of $2n$ as we now describe.

Recall that in the centrally symmetric annulus, the date line is the vertical line at the bottom of the annulus.
Passing to the quotient, the \newword{date line} in the two-orbifold disk is a vertical segment downward from the bottom orbifold point that is shown in gray in the right picture of Figure~\ref{fold ann fig}.
Recall also that the \newword{toggle line} is the vertical line connecting the two orbifold points.

To read a cycle from the boundary of a block $E$ of $\P$, we start at some numbered point on the boundary of $E$ and follow the boundary, keeping the interior of $E$ on the right.
Up to isotopy, we can assume that the boundary crosses the toggle line $0$, $1$, or $2$ times.
Each time we reach a numbered point, we take that point or its negative, taking the negative if and only if we have crossed the toggle line once but not a second time.
Then we add $2wn$, with $w$ defined as follows:
We begin with $w=0$.
Each time we cross the date line, we add or subtract $1$ to $w$.
Before crossing the toggle line, and after crossing the toggle line twice, clockwise crossings of the date line add $1$ and counterclockwise crossings of the date line subtract~$1$.
When the toggle line has been crossed once by not twice, clockwise crossings of the date line subtract $1$ and counterclockwise crossings of the date line add~$1$.
When we return to the numbered point where we started, if we record the starting point again, then we have defined a finite cycle.  
(This happens if and only if the block contains no orbifold points.)
If we record the negative of the starting point plus a zero or nonzero multiple of $2n$, then we read around the boundary a second time to produce a finite cycle.
(This happens if and only if the block contains exactly one orbifold point.)
If we record the starting point plus a nonzero multiple of $2n$, then we have defined an infinite cycle.
(This happens if and only if the block contains both orbifold points.)
Thus we record a cycle that, together with its negation and all mod-$2n$ translates, is part of the cycle notation for $\perm^C(\P)$. 

\begin{example}\label{perm ex C} 
Labeling the noncrossing partitions shown in Figure~\ref{nc C exs} as $\P_1,\P_2,\P_3$ from left to right and labeling the noncrossing partition in Figure~\ref{one more nc C ex} as $\P_4$, we apply $\perm^C$ to obtain the following permutations in $\afftype{C}_6$:
\begin{align*}
\perm^C(\P_1)&=(\!(\cdots\,1\,\,\,5\,\,\,11\,\,\,15\,\cdots)\!)\,(\!(2)\!)_{14}\,(\!(4\,\,\,6)\!)_{14}\\
\perm^C(\P_2)&=(1\,\,-1)_{14}\,(2\,\,\,8\,\,\,12\,\,\,6)_{14}\,(\!(3\,\,\,4)\!)_{14}\,(\!(5)\!)_{14}\\
\perm^C(\P_3)&=(\!(1\,\,-2)\!)_{14}\,(\!(5\,\,\,8\,\,\,4\,\,\,3)\!)_{14}\\
\perm^C(\P_4)&=(\!(1\,\,\,2)\!)_{14}\,(\!(3)\!)_{14}\,(\!(4\,\,\,33)\!)_{14}\,(\!(6\,\,\,36)\!)_{14}
\end{align*} 
\end{example}

We conclude this section with a brief discussion of Kreweras complements in type $\afftype{C}$.
The Kreweras complement $\Krew$ on the symmetric annulus commutes with the symmetry $\phi$, and thus restricts to an anti-automorphism of $\tNCAphic$.
By Theorem~\ref{Krew thm}, we have $\perm(\Krew(\P))=\perm(\P)^{-1}c$ for $\P\in\tNCAphic$.
We can describe the map $\Krew^C$ on noncrossing partitions of the two-orbifold disk such that $\perm^C(\Krew^C(\P))=\perm^C(\P)^{-1}c$:
For each point $i$ in the signing, choose a point~$i'$ on the boundary of $C$ between $i$ and $c(i)$.
(The point $c(i)$ is the next numbered point on the boundary clockwise from $i$.)
To construct $\Krew^C(\P)$, we first make the noncrossing partition~$\P'$ of $C$ with numbered points $1',\ldots,n'$ whose blocks are the maximal embedded blocks disjoint from the nontrivial blocks of $\P$, together with trivial blocks $\set{i'}$ for every $i'$ contained in a block of $\P$.
Each orbifold point is contained in a block of $\P'$ if and only if it is not contained in a block of~$\P$.
We obtain $\Krew(\P)$ from $\P'$ by applying a homeomorphism that fixes each orbifold point and rotates each~$i'$ back to $i$, moving the interior of $C$ as little as possible.

\begin{example}\label{Krew C ex}
Figure~\ref{Krew C exs} shows the Kreweras complements of the noncrossing partitions of the two-orbifold disk shown in Figure~\ref{nc C exs}.
\end{example}

\begin{figure}
\begin{tabular}{cccc}
\scalebox{0.9}{\includegraphics{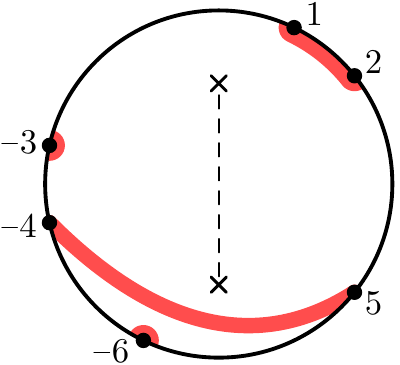}}
&\scalebox{0.9}{\includegraphics{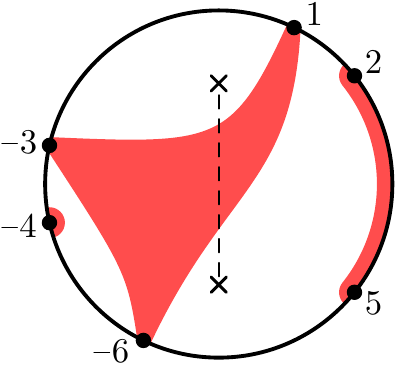}}
&\scalebox{0.9}{\includegraphics{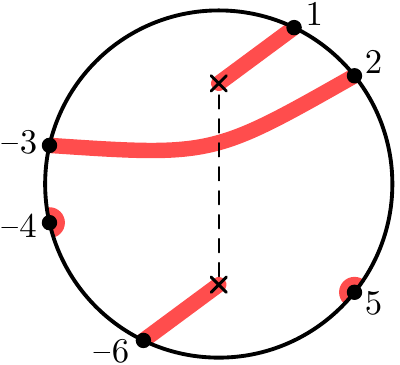}}
\end{tabular}
\caption{Kreweras complements of the noncrossing partitions from Figure~\ref{nc C exs}}
\label{Krew C exs}
\end{figure}

\section{Factoring translations}\label{McSul sec}
In this section, we connect and compare our results on noncrossing partitions of affine types $\afftype A$ and $\afftype{C}$ to the results of McCammond and Sulway~\cite{McSul}.
In particular, we highlight the role that the \emph{factored translations} of \cite{McSul} play in motivating the decision to allow \emph{dangling annular blocks} in our construction.
We also show that the combinatorial models lead to a different way of factoring translations that is in some sense more natural.

\subsection{Factored translations in general affine type}\label{McSul general sec}
We begin with a brief summary of background material related to factored translations.
Unjustified assertions in the background material that follows are proved in \cite{McFailure,McSul}.

McCammond and Sulway work with a fixed affine reflection representation of an affine Coxeter group.  
We now explain the affine representation in our context.

When $W$ is a Coxeter group of affine type, the dual representation restricts to an action on the affine hyperplane $\set{x\in V^*:\br{x,\delta}=1}$, acting by Euclidean motions.
We refer to this affine hyperplane as $E$ and write $E_0$ for the linear hyperplane $\set{x\in V^*:\br{x,\delta}=0}$ parallel to $E$.
The linear hyperplane $E_0$ can be identified with the span of the root system $\Phi_\fin$ by identifying a root $\beta\in\Phi_\fin$ with $K(\,\cdot\,,\beta)$.

In the action of $W$ on $E$, a reflection $t\in W$ associated to a root $\phi=\beta+k\delta$, with $\beta\in\Phi_\fin$, fixes the set $\set{x\in E:\br{x,\beta}=k}$, which is an $(n-2)$-dimensional affine subspace of $E$, and thus an affine hyperplane in $E$.  
The action of $t$ also negates the vector $K(\,\cdot\,,\beta)\in E_0$.

Given a Coxeter system $(W,S)$, the associated \newword{Artin group} is presented as the group generated by $S$, subject to the braid relations that define $W$ but without the relations $s^2=1$ for $s\in S$.
If $W$ is finite or affine, then the associated Artin group is called \newword{spherical} or \newword{Euclidean} respectively.
When $W$ is finite, the interval $[1,c]_T$ is a lattice \cite{BWlattice}, and the lattice property is the crucial point in proving various desirable properties of the Artin group \cite{Bessis,Bra-Wa}.
When $W$ is affine, $[1,c]_T$ need not be a lattice.
McCammond and Sulway \cite{McSul} (building on work of Brady and McCammond \cite{Br-McC,McFailure} and earlier work of Digne \cite{Digne1,Digne2}) 
proved the same desirable properties for Euclidean Artin groups by extending $W$ to a larger ``crystallographic group'' and showing that the analog of $[1,c]_T$ in the crystallographic group is a lattice.

The \newword{translations} in $W$ are the elements that act on $E$ as translations.
In the cases where $[1,c]_T$ is not a lattice, McCammond and Sulway construct the larger group by ``factoring'' some of the translations in $W$ and forming the group generated by $W$ and the factors.
We now explain their factoring scheme.

Given a Coxeter element $c$, the \newword{Coxeter axis} in $E$ is the unique line in $E$ fixed (as a set) by the action of $c$.
This is the intersection of $E$ with the Coxeter plane in~$V^*$, defined in Section~\ref{aff sec}.
Recalling that the Coxeter plane is spanned by $\omega_c(\delta,\,\cdot\,)$ and $\omega_c(\gamma_c,\,\cdot\,)$, and noting that $\omega_c(\delta,\,\cdot\,)$ is parallel to $E$ because $\br{\omega_c(\delta,\,\cdot\,),\delta}=\omega_c(\delta,\delta)=0$, we see that the direction of the Coxeter axis is $\omega_c(\delta,\,\cdot\,)$.

A \newword{horizontal reflection} is a reflection in $W$ whose reflecting hyperplane in $E$ is parallel to the Coxeter axis.
Thus, if $\beta+k\delta$ is a root in $\Phi$ with $\beta\in\Phi_\fin$, then the reflection orthogonal to $\beta+k\delta$ is horizontal if and only if $\set{x\in E:\br{x,\beta}=k}$ is parallel to $\omega_c(\delta,\,\cdot\,)$, which is if and only if $\omega_c(\delta,\beta)=0$.
The set of roots $\beta\in\Phi_\fin$ such that $\omega_c(\delta,\beta)=0$ is a finite root system called the \newword{horizontal root system}.
For every positive root $\beta$ in the horizontal root system, there are exactly two horizontal reflections in $[1,c]_T$, the reflection orthogonal to $\beta$ and the reflection orthogonal to~$\delta-\beta$.
By contrast, every non-horizontal reflection is in $[1,c]_T$.

A root system is \newword{reducible} if and only if it can be partitioned into two nonempty subsets such that each root in one subset is orthogonal to each root in the other subset.
Otherwise, the root system is \newword{irreducible}.
For $W$ affine, $[1,c]_T$ is a lattice if and only if the horizontal root system is irreducible. 
In any case, if the irreducible components of the horizontal root system are $\Psi_1,\ldots,\Psi_k$, then write $U_i$ for the span of $\set{K(\,\cdot\,,\beta):\beta\in\Psi_i}$.
Writing also $U_0$ for the span of $\omega_c(\delta,\,\cdot\,)$, the linear subspace $E_0$ parallel to $E$ is an orthogonal direct sum ${U_0\oplus\cdots\oplus U_k}$.

The subgroup of $W$ consisting of translations is generated by products of two reflections in $W$ with adjacent parallel reflecting hyperplanes.
If $w$ is a translation in $[1,c]_T$, with translation vector $\lambda$, write $\lambda=\lambda_0+\lambda_1+\cdots+\lambda_k$ with $\lambda_i\in U_i$ for $i=0,\ldots,k$.
Fixing real numbers $q_1,\ldots,q_k$ with $\sum_{i=1}^kq_k=1$, we define $k$ distinct \newword{factored translations} associated to~$w$, namely the translations by $\lambda_i+q_i\lambda_0$ for $i=1,\ldots,k$.
The composition of these $k$ factored translations is~$w$.
Write $F$ for the set of all factored translations arising from all translations $w\in[1,c]_T$ and write $F^{\pm1}$ for the set of all factored translations and their inverses.

These translations on $E$ uniquely determine linear transformations on $V^*$, which in turn determine dual linear transformations on $V$.
Identifying elements of $F$ with these linear maps, the set $T\cup F$ generates a supergroup of $W$ (as it acts on~$V^*$ and on~$V$).
Define a length function $\ell_{T\cup F}$ on the supergroup by setting reflections to have length $1$ and factored translations to have length $\frac2k$ and computing lengths of words accordingly.
Using this length function, we define a partial order on the supergroup analogous to the absolute order on $W$ and write $[1,c]_{T\cup F}$ for the interval between $1$ and $c$ in this partial order.
(In type $\afftype{A}$, $k$ is either $1$ or $2$, so if we factor translations, there are no fractional lengths.)
Each edge $u\covered w$ in the interval is labeled with $u^{-1}w$, which is an element of $T\cup F^{\pm1}$.

The \newword{interval group} associated to $[1,c]_{T\cup F}$ is the group presented abstractly by the subset of $T\cup F$ consisting of labels that appear in $[1,c]_{T\cup F}$, subject to relations that equate the label sequences on different unrefinable chains with the same endpoints.

In \cite{McSul}, the constants $q_i$ are always set to $\frac1k$ (so, in fact, constants $q_i$ are not mentioned explicitly), and the resulting interval group is the crystallographic group mentioned above.
Guided by the combinatorics of noncrossing partitions of the annulus, in Section~\ref{McSul A sec}, we make a different choice of the $q_i$ in type~$\afftype{A}$.
Thus we are interested in the following theorem.

\begin{thm}\label{same interval group}
Given two choices $q_1+\cdots+q_k=1$ and $q'_1+\cdots+q'_k=1$, let~$F$ and~$F'$ be the corresponding sets of factored translations.
Then the intervals $[1,c]_{T\cup F}$ and $[1,c]_{T\cup F'}$ are isomorphic as labeled posets.
Thus the interval group constructed from $[1,c]_{T\cup F}$ is isomorphic to the interval group constructed from $[1,c]_{T\cup F'}$.
\end{thm}

In preparation for the proof of Theorem~\ref{same interval group}, we quote and rephrase some results of~\cite{McSul}.
We begin with following proposition, which arises by combining facts from~\cite{McSul}.
First, \cite[Proposition~6.3]{McSul} says that every translation in $[1,c]_T$ is part of a ``horizontal factorization'' of $c$.
Then \cite[Definition 6.6]{McSul} defines a ``diagonal translation'' to be a translation that occurs in a horizontal factorization and gives a simple argument that every diagonal translation satisfies the conclusion of the following proposition (which we state without this intermediate terminology).

\begin{prop}\label{trans nontriv}
For every translation in $[1,c]_T$ and every $i=0,\ldots,k$, translation vector has nonzero projection onto the component $U_i$ of the decomposition $E_0=U_0\oplus\cdots\oplus U_k$.
\end{prop}

As before, write $\Psi_1,\ldots,\Psi_k$ for the irreducible components of the horizontal root system and $U_i$ for the span of $\set{K(\,\cdot\,,\beta):\beta\in\Psi_i}$.
\begin{prop}\label{fact commute}
If $x$ is a reflection associated to $\Psi_i$ or a factored translation associated to $U_i$ and $y$ is a reflection associated to $\Psi_j$ or a factored translation associated to $U_j$ for $j\neq i$, then $x$ and $y$ commute.
\end{prop}
\begin{proof}
All translations commute with each other.
Horizontal reflections associated to different irreducible components of the horizontal root system have orthogonal reflecting hyperplanes, so they commute with each other.
If $f$ is a factored translation associated to $U_i$ and $r$ is a horizontal reflection associated to $\Psi_j$ for $j\neq i$, then the translation vector of $f$ is in the span of $\set{K(\,\cdot\,,\beta):\beta\in\Psi_i}\cup\set{\omega_c(\delta,\,\cdot\,)}$.
(Recall that $\omega_c(\delta,\,\cdot\,)$ is the direction of the Coxeter axis and spans $U_0$.)
The reflecting hyperplane for $r$ contains this span:
It contains $\omega_c(\delta,\,\cdot\,)$ because $r$ is horizontal, and it contains $\set{K(\,\cdot\,,\beta):\beta\in\Psi_i}$ because every root in $\Psi_i$ is orthogonal (in the sense of $K$) to every root in $\Psi_j$.
\end{proof}

The following proposition gathers results of~\cite{McSul} on reduced words for $c$ in the alphabet $T\cup F^{\pm}$ (or in other words sequences of labels on maximals chains in $[1,c]_{T\cup F}$), also pointing out that the results hold for arbitrary choices of the~$q_i$.

\begin{prop}\label{if a fact}
For an arbitrary choice of real numbers $q_1,\ldots,q_k$ summing to~$1$, suppose a factored translation appears as one of the letters in a reduced word for $c$ in the alphabet $T\cup F^{\pm}$.
\begin{enumerate}[\quad\rm\bf1.]
\item \label{fact horiz}
Each reflection in the word is horizontal.
\item \label{ref in int}
Each reflection in the word is contained in $[1,c]_T$.
\item \label{n-2 ref}
There are exactly $n-2$ reflections in the word.
\item \label{k trans} 
There are exactly $k$ translations in the word.
\item \label{all factors}
The translations in the word are the $k$ factors of some translation in $[1,c]_T$.
(In particular, they are elements of $F$, whereas \textit{a priori} they are in $F^\pm$.)
\item \label{reorder}
The reduced word can be reordered, by swapping letters that commute in the group, so that, for each~$i$, all reflections associated to $\Psi_i$ and all factored translations associated to $U_i$ are adjacent to each other.
\item \label{fact c}
For each $i$, let $c_i$ be the product of the subword consisting only of reflections associated to $\Psi_i$ and all factored translations associated to $U_i$.
Then $c_i$ depends only on $c$ and the constants $q_i$, not on the choice of reduced word. 
\end{enumerate}
\end{prop}
In Proposition~\ref{if a fact}.\ref{fact horiz}, there is no assertion that the reflections are in the finite Coxeter group associated to the horizontal root system, but only that they are orthogonal to roots $\beta+k\delta$ for $\beta$ in the horizontal root system.
\begin{proof}
Assertions~\ref{fact horiz} and~\ref{n-2 ref} are proved as part of the proof of \cite[Lemma~7.2]{McSul}.
Formally, the proof there is given in the case where (in our notation) $q_i=\frac1k$ for all $i$, but the argument given there is independent of that choice.
(In the statement here, we correct for different definitions of $n$ between the two papers.)
Assertion~\ref{k trans} is not stated explicitly in that proof but is immediate because each factored translation has fractional length (or ``weight'') $\frac2k$ in the definition of $[1,c]_{T\cup F}$ given in~\cite{McSul}.

We can move the $k$ letters that are factored translations to be next to each other at the beginning of the word by replacing instances of $t,f$ by $f,f^{-1}tf$ for factored translations $f$ and reflections~$t$.
We obtain an expression for $c$ as a translation~$\tau$ followed by a product of reflections.
In particular, $\tau\in[1,c]_T$ and the reflections are all in $[1,c]_T$.  
This proves Assertion~\ref{ref in int}.

The horizontal reflections fix the Coxeter axis, so Assertion~\ref{fact horiz} implies that the component of $\tau$ in the direction of the Coxeter axis equals the translation component of the action of the Coxeter element.
For the same reason, this component is the same as the $U_0$-component of any translation in $[1,c]_T$.

Proposition~\ref{trans nontriv} says that the translation vector $\lambda$ for $\tau$ has nontrivial projections to each of the components $U_1,\ldots,U_k$ of $E_0$.
Since each factored translation has nonzero projection to $U_i$ for exactly one $i$ in $\set{1,\ldots,k}$, we see that the $k$ factored translations in the word are associated to the $k$ different components $U_1,\ldots,U_k$.
Since these factored translation vectors add up to $\lambda$, they must in fact be the factors $\lambda_i+q_i\lambda_0$ of $\lambda$.
We have proved Assertion~\ref{all factors}.

Assertion~\ref{reorder} is immediate by Proposition~\ref{fact commute}.

The Coxeter element decomposes as a translation $\lambda_0$ along the Coxeter axis and an action in the directions $U_1\oplus\cdots\oplus U_k$ orthogonal to the Coxeter axis.
In the factorization $c=c_1\cdots c_k$, each $c_i$ is a translation $q_i\lambda_0$ along the Coxeter axis composed with a motion that only moves points in $U_i$.
Thus the choice of fixed $q_i$ and $c$ determines $c_i$ completely as the translation $q_i\lambda_0$ composed with the action of $c$ on~$U_i$.
This is Assertion~\ref{fact c}.
\end{proof}

Now we can prove that the choice of the $q_i$ does not affect the interval group.

\begin{proof}[Proof of Theorem~\ref{same interval group}]
Let $p:F\to F'$ be the map that sends a factored translation in $F$ with vector $\lambda_i+q_i\lambda_0$ to the factored translation in $F'$ with vector $\lambda_i+q'_i\lambda_0\in F'$ for the same~$i$.
This map is well-defined because the translation vector for a factored translation $f\in F$ has a nonzero projection to $U_i$ for exactly one~$i$.

We extend $p$ to a map on label sequences on maximal chains in $[1,c]_{T\cup F}$.
The map changes each factored translation by a multiple of $\omega_c(\delta,\,\cdot\,)$ (the vector spanning $U_0$).
By Proposition~\ref{if a fact}.\ref{all factors}, the total change to all translation vectors is zero.
By Proposition~\ref{if a fact}.\ref{fact horiz}, the changes to the factored translations commute with the reflections in the sequence, and thus cancel each other out.
We see that the map outputs the label sequence on a maximal chain in $[1,c]_{T\cup F'}$.
The analogous map in the opposite direction is the inverse, so the map is a bijection.
This bijection on maximal chains induces an isomorphism of labeled posets from the interval $[1,c]_{T\cup F}$ to the interval $[1,c]_{T\cup F'}$.
The assertion about isomorphism of interval groups follows.
\end{proof}

We will see in Section~\ref{McSul A sec} that the choice of constants $q_i$ suggested by the combinatorics of noncrossing partitions is the unique choice such that $T\cup F^{\pm1}$ is closed under conjugation by elements of the supergroup.
We now establish some general criteria for determining whether $T\cup F^{\pm1}$ is closed.
The following lemma is an immediate consequence of the fact that $T$ is closed under conjugation by elements of $T$ and $F^\pm$ is closed under conjugation by elements of $F$.

\begin{lemma}\label{amounts to}
$T\cup F^{\pm}$ is closed under conjugation in the group generated by $T\cup F$ if and only if $T$ is closed under conjugation by elements of $F$ and $F^{\pm1}$ is closed under conjugation by elements of $T$.
\end{lemma}

The \newword{fundamental co-weights} $\rho\ck_1,\ldots,\rho\ck_n$ are the basis of $V^*$ that is dual to the simple roots $\alpha_1,\ldots,\alpha_n$.

\begin{lemma}\label{closure and weights}
Suppose $W$ is an affine Coxeter group and $f$ is a translation in $W$ with translation vector~$\lambda$.
Then $T$ is closed under conjugation by $f$ if and only if $\lambda$ has integer fundamental co-weight coordinates.
\end{lemma}
\begin{proof}
Each reflection $t\in T$ is orthogonal to some root $\beta+k\delta$ for $\beta\in\Phi_\fin$ and $k\in\integers$.
The fixed hyperplane of $t$ is $\set{x\in E:\br{x,\beta}=k}$.
The fixed hyperplane of the reflection $ftf^{-1}$ is the hyperplane $\set{x\in E:\br{x,\beta}=k+\br{\lambda,\beta}}$.
This is the reflecting hyperplane for the reflection orthogonal to $\beta+(k+\br{\lambda,\beta})\delta$, which is in $T$ if and only if $\beta+(k+\br{\lambda,\beta})\delta$ is a root, equivalently if and only if $\br{\lambda,\delta}$ is an integer.
If $\lambda$ has integer fundamental co-weight coordinates, then $\br{\lambda,\beta}$ is an integer.
On the other hand, if $\lambda=\sum c_k\rho_k\ck$ and some $c_i$ is not an integer, then the reflection~$s_i$ in $W_\fin$ is the reflection orthogonal to $\alpha_i$ and $\br{\lambda,\alpha_i}=c_i$ is not an integer.
\end{proof}

\begin{lemma}\label{closure and reflections}
Let $Q$ be a set of translations of the affine space $E$.
The following are equivalent:
\begin{enumerate}[\quad(i)]
\item The set $Q$ is closed under conjugation by $T$.
\item For every $w\in W$, the action of $w$ on the linear subspace $E_0$ permutes the set of translation vectors for $Q$.
\item For every $s\in S_\fin$, the action of $s$ on $E_0$ permutes the set of translation vectors for $Q$.
\end{enumerate}
\end{lemma}
\begin{proof}
Consider a translation $q$ on $E$ with translation vector $\lambda\in E_0$ and consider a rigid motion $m$ on $E$.
The motion $m$ extends uniquely to a linear map $\mu$ on $V^*$, and this linear map must fix the plane $E_0$ as a set (or else, by linearity, it cannot fix the plane $E$ as a set).
For any $x\in E$, we have $mqm^{-1}(x)=\mu(\mu^{-1}(x)+\lambda)=x+\mu(\lambda)$.
Thus the conjugation of $q$ by $m$ is a translation with translation vector $\mu(\lambda)\in E_0$.

Each reflection $t\in T$ is a linear map on $V^*$ that restricts to a rigid motion on $E$ and thus conjugating by $t$ changes a translation with vector $\lambda$ to a translation with vector $t(\lambda)$.
We see that $Q$ is closed under conjugation by $T$ if and only if each $t\in T$ permutes the set of translation vectors for $Q$.
Since $T$ generates $W$, this is if and only if each $w\in W$ permutes the set of translation vectors for $Q$.

Every reflection in $W$ acts on $E$ as an affine reflection whose reflecting hyperplane is parallel to the reflecting hyperplane for some reflection in $W_\fin$.
Parallel reflections restrict to the same reflection on $E_0$.
Thus every $t\in T$ permutes the set of translation vectors for $F$ if and only if every reflection in $W_\fin$ permutes the set of translation vectors.
Since $S_\fin$ generates $W_\fin$, this is if and only if every $s\in S_\fin$ permutes the set of translation vectors.
\end{proof}

\subsection{Factored translations in affine type $\afftype{A}$}\label{McSul A sec}
In this section, we relate the construction of $\tNCAc$ to the McCammond-Sulway factored translation construction in type $\afftype{A}_{n-1}$.
In what follows, we will use without comment the fact that in type~$\afftype{A}$, roots and coroots are the same, and weights and co-weights are the same. 

In this case, the orthogonal decomposition $E_0=U_0\oplus\cdots\oplus U_k$ has $k=1$ or~$2$ (or $k=0$ if $n=2$).
Given a choice of $c$, we will define subspaces $U_\out$ and $U_\inn$ such that the decomposition is always $E_0=U_0\oplus U_\out\oplus U_\inn$, but possibly $U_\out$ or $U_\inn$ is the zero subspace.
Similarly, for the purposes of factoring translations, we will write $q_\out$ and $q_\inn$ for the numbers $q_1,\ldots,q_k$.
We write $\#\out$ for the number of outer points on the annulus corresponding to $c$ and $\#\inn$ for the number of inner points.
We will prove the following theorems, which tie the work of this paper to the results of McCammond and Sulway~\cite{McSul}.

\begin{thm}\label{F L}
Let $c$ be a Coxeter element of $\Stilde_n$, represented as a partition of $\set{1,\ldots,n}$ into inner points and outer points.
Fix numbers $q_\out=\frac{\#\inn}{\#\out+\#\inn}$ and ${q_\inn=\frac{\#\out}{\#\out+\#\inn}}$.
\begin{enumerate}[\quad\rm\bf1.]
\item \label{F loop}
The set of factored translations is $F=\set{\ell_i:i\text{ is outer}}\cup\set{\ell_i^{-1}:i\text{ is inner}}$.
\item \label{supergroup}
The supergroup of $\Stilde_n$ generated by $T\cup F$ is $\SZmodn$.
\item \label{affine supergroup}
On the affine space $E$, the supergroup is generated by $W$ and the translations by vectors in $E_0$ with integer fundamental-weight coordinates.
\item  \label{F L intervals}
The interval $[1,c]_{T\cup F}$ coincides with $[1,c]_{T\cup L}$, which is isomorphic to the lattice $\tNCAc$.
\item \label{unique conj}
This choice of $q_\out$ and $q_\inn$ is the unique choice such that $T\cup F^{\pm1}$ is closed under conjugation by elements of the supergroup generated by $T\cup F$.
\end{enumerate}
\end{thm}

\begin{remark}\label{one point case}
When the choice of $c$ corresponds to exactly one inner point or exactly one outer point, the horizontal root system is irreducible, so $[1,c]_T$ is a lattice, and thus we do not need factored translations in the setup of \cite{McSul}.
However, we can factor translations anyway, and Theorem~\ref{F L} is still valid in this special case.  
\end{remark}

\begin{thm}\label{interval group T L}
For any choice of $q_\out+q_\inn=1$, the interval $[1,c]_{T\cup F}$ is isomorphic as a labeled poset to $[1,c]_{T\cup L}$.
Thus the interval group constructed from $[1,c]_{T\cup F}$ is isomorphic to the interval group constructed from $[1,c]_{T\cup L}$.
\end{thm}

%

To prove these theorems, we begin by determining the horizontal reflections.
The following proposition is an easy consequence of Proposition~\ref{om del e}, Theorem~\ref{isom}, and the fact that $\Phi_\fin=\set{\pm(\e_j-\e_i):1\le i<j\le n}$.

\begin{prop}\label{A horiz}
Let $c$ be a Coxeter element of $\Stilde_n$, represented as a partition of $\set{1,\ldots,n}$ into inner points and outer points.
Then
\begin{enumerate}[\quad\rm\bf1.]
\item \label{A horiz root}
The horizontal root system is the set of vectors $\pm(\e_j-\e_i)$ such that $1\le i<j\le n$ and $i$ and $j$ are either both inner or both outer.
\item \label{all horiz ref}
The horizontal reflections in $\Stilde_n$ are the reflections of the form $(i\,\,\,j+kn)_n$ for $1\le i<j\le n$ such that $i$ and $j$ are either both inner or both outer and $k\in\integers$.
\item \label{A horiz ref}
The horizontal reflections in $[1,c]_T$ are the reflections of the form $(i\,\,\,j)_n$ or $(i\,\,\,j-n)_n$ for $1\le i<j\le n$ such that $i$ and $j$ are either both inner or both outer.
\item\label{horiz NC}
The horizontal reflections in $[1,c]_T$ are the elements $\perm(\P)$ such that the only nontrivial block of $\P$ is an arc or boundary segment either connecting two distinct inner points or connecting two distinct outer points.
\end{enumerate}
\end{prop}

We next characterize the translations in $[1,c]_T$.

\begin{prop}\label{A trans}
Let $c$ be a Coxeter element of $\Stilde_n$, represented as a partition of $\set{1,\ldots,n}$ into inner points and outer points.
The translations in $[1,c]_T$ are the permutations $(\cdots\,i\,\,\,i+n\,\cdots)(\cdots\,j\,\,\,j-n\,\cdots)$ with $i$ outer and $j$ inner.
The corresponding translation vector on the affine space $E$ is $(\rho_i-\rho_{i-1})-(\rho_j-\rho_{j-1})$, with indices interpreted modulo $n$, and the corresponding noncrossing partition of the annulus has exactly one nontrivial block, an annulus containing $i$ and $j$ and no other numbered point.
\end{prop}
\begin{proof}
The subgroup of $\Stilde_n$ consisting of translations is generated by elements of the form $\tau_{ij}=(i\,\,\,j)_n(i\,\,\,j+n)_n=(\cdots\,i\,\,\,i+n\,\cdots)(\cdots\,j\,\,\,j-n\,\cdots)$ for $i\neq j\mod n$.
One easily checks that a composition of such elements is $\perm(\P)$ for some $\P\in\tNCAc$ if and only if it is a single element $\tau_{ij}$ with $i$ outer and $j$ inner.
Since we know how $\tau_{ij}$ acts on the basis $\e_1,\ldots,\e_{n+1}$ of $\reals^{n+1}$, we can compute its action on simple roots:
\[\alpha_k\mapsto\alpha_k
+\begin{rcases}\begin{dcases}
\delta&\text{if }k=i-1\\
-\delta&\text{if }k=i,\text{ or}\\
0&\text{otherwise}
\end{dcases}\end{rcases}
+\begin{rcases}\begin{dcases}
-\delta&\text{if }k=j-1\\
\delta&\text{if }k=i,\text{ or}\\
0&\text{otherwise}
\end{dcases}\end{rcases}.
\]
Since $\delta$ is $\sum_{i=1}^n\alpha_i$, the matrix for $\tau_{ij}$ acting on column vectors of simple root coordinates is the identity matrix $\pm$ four matrices that consist of a column of $1$'s, with $0$'s elsewhere.
The same matrix acting on the right on row vectors of fundamental weight coordinates gives the dual action of $\tau_{ij}^{-1}$ on $V^*$.
The action is $\rho_k\mapsto\rho_k+\rho_{i-1}-\rho_i-\rho_{j-1}+\rho_j$, so $\tau_{ij}$ is a translation by $(\rho_i-\rho_{i-1})-(\rho_j-\rho_{j-1})$ on vectors in $E=\set{x\in V^*:\br{x,\delta}=1}$.
\end{proof}

Each translation in $[1,c]_T$ has an obvious factorization into a loop $\ell_i$ for $i$ outer and a loop $\ell_j^{-1}$ for $j$ inner.
The proof of Proposition~\ref{A trans} also shows that this factorization of group elements corresponds to factoring the corresponding translation vector $\lambda$ into $\rho_i-\rho_{i-1}$ and $-(\rho_j-\rho_{j-1})$.
To complete the proof of Theorem~\ref{F L}.\ref{F loop}, we will show that these factors are $\lambda_\out+q_\out\lambda_0$ and $\lambda_\inn+q_\inn\lambda_0$.

For each $i$, define $\sigma_i=\rho_i-\rho_{i-1}$, taking indices modulo $n$.
Define $U_\out$ to be the subspace $\set{\sum c_i\sigma_i:c_i\in\reals,\sum c_i=0}$ of $E_0$, where both sums are over the outer points in $\set{1,\ldots,n}$.
Define $U_\inn$ the same way, but taking both sums over inner points instead.
If there is exactly one outer point, then $U_\out=\set{0}$, and if there is exactly one inner point, then $U_\inn=\set{0}$.

\begin{proposition}\label{A subspace decomp}
The subspaces $U_\out$ and $U_\inn$ are spans of sets $\set{K(\,\cdot\,,\beta):\beta\in\Psi}$ for (possibly trivial) irreducible components $\Psi$ of the horizontal root system.
\end{proposition}
\begin{proof}
Proposition~\ref{A horiz} implies that there is an irreducible component of the horizontal root system spanned by vectors $\e_j-\e_i=\alpha_i+\cdots+\alpha_{j-1}$ for $1\le i<j\le n$, both outer.
The vector $K(\,\cdot\,,\alpha_i+\cdots+\alpha_{j-1})$ has fundamental-weight coordinates given by multiplying the Cartan matrix $A$ on the left of a column vector with $1$ in positions $i$ through $j-1$ and $0$ elsewhere.
This is $\sigma_i-\sigma_j$.
It follows that $U_\out$ is the span of $\set{K(\,\cdot\,,\beta):\beta\in\Psi}$ for an irreducible component $\Psi$.
The same argument applies to $U_\inn$.
\end{proof}

As a consequence of Proposition~\ref{A subspace decomp}, the orthogonal direct product decomposition used to factor the translations is $E_0=U_0\oplus U_\out\oplus U_\inn$.

\begin{proposition}\label{A trans decomp}
Suppose $i$ is outer and $j$ is inner, and let $\lambda=\sigma_i-\sigma_j$, so that the translation $x\mapsto x+\lambda$ on $E$ is in $[1,c]_T$.
Write $a$ for $\frac12\omega_c(\delta,\,\cdot\,)$ and define 
\begin{align*}
\lambda_0&=\left(\frac1{\#\out}+\frac1{\#\inn}\right)a\\
\lambda_\out&=\sigma_i-\frac1{\#\out}a\\
\lambda_\inn&=-\sigma_j-\frac1{\#\inn}a
\end{align*}
The decomposition $\lambda=\lambda_0+\lambda_\out+\lambda_\inn$ has
$\lambda_0\in U_0$, $\lambda_\out\in U_\out$, and $\lambda_\inn\in U_\inn$.
\end{proposition}
\begin{proof}
First, $\lambda_0\in U_0$ because $U_0$ is the span of $a$.
Proposition~\ref{om del rho} says that $a=\sum_{k\text{ outer}}\sigma_k$, so $\lambda_\out$ is a linear combination of the $\sigma_k$ for $k$ outer with the sum of coefficients equal to $0$. Thus $\lambda_\out\in U_\out$.
Proposition~\ref{om del rho} also says that $a=-\sum_{k\text{ inner}}\sigma_k$, and we argue similarly that $\lambda_\inn\in U_\inn$.
\end{proof}

\begin{proposition}\label{A fact}
Let $c$ be a Coxeter element of $\Stilde_n$, represented as a partition of $\set{1,\ldots,n}$ into inner points and outer points.
For $q_\out$ and $q_\inn$ as in Theorem~\ref{F L},
the factored translations have translation vectors $\pm(\rho_i-\rho_{i-1})$ for $i=1,\ldots,n$, taking the $+$ if $i$ is outer or the $-$ if $i$ is inner.
\end{proposition}
\begin{proof}
Proposition~\ref{A trans} says that the translation vectors for translations in $[1,c]_T$ are $\sigma_i-\sigma_j$ for $i$ outer and $j$ inner.
Using Proposition~\ref{A trans decomp}, we compute that $\lambda_\out+q_\out\lambda_0=\sigma_i$ and $\lambda_\inn+q_\inn\lambda_0=-\sigma_j$.
\end{proof}

We have proved Theorem~\ref{F L}.\ref{F loop}.
Now Theorem~\ref{F L}.\ref{supergroup}--\ref{F L intervals} follow because $T\cup L$ generates $\SZmodn$, because the group of translations by vectors in $E_0$ with integer fundamental weight coordinates is generated by the vectors $\sigma_i=\rho_i-\rho_{i-1}$, and because of Theorem~\ref{isom}.

\begin{proof}[Proof of Theorem~\ref{F L}.\ref{unique conj}]
For $q_\out$ and $q_\inn$ as in Theorem~\ref{F L}, Proposition~\ref{A fact} implies that every element of $F$ has integer fundamental weight coordinates.
Furthermore, we calculate that a simple reflection either fixes $\rho_i-\rho_{i-1}$, sends it to $\rho_{i-1}-\rho_{i-2}$, or sends it to $\rho_{i+1}-\rho_i$.
Thus by Lemmas~\ref{amounts to}, \ref{closure and weights}, and \ref{closure and reflections}, $T\cup F^{\pm1}$ is closed under conjugation in the supergroup.  

Now consider any choice of $q'_\out$ and $q'_\inn$ summing to~$1$, let $F'$ be the associated set of factored translations, and suppose $T\cup (F')^{\pm1}$ is closed under conjugation.

Since fundamental weights and co-weights coincide in type~$\afftype{A}$, Lemma~\ref{closure and weights} says that every translation vector for an element of $F'$ has integer fundamental weight coordinates.
Continue the notation $a=\frac12\omega_c(\delta,\,\cdot\,)$.
Suppose~$\lambda$ is the translation vector for some translation in $[1,w]_T$, suppose~$\lambda$ factors as ${\lambda_\out+\lambda_\inn}$ with respect to $q_\out$ and $q_\inn$, and suppose $\lambda$ factors as ${\lambda'_\out+\lambda'_\inn}$ with respect to $q'_\out$ and $q'_\inn$.
Then $\lambda'_\out=\lambda_\out+ka$ and $\lambda'_\inn=\lambda_\inn-ka$, for some $k\in\rationals$ that is independent of which translation in $[1,c]_T$ was chosen.
By Proposition~\ref{A fact}, $(F')^{\pm1}$ consists of translations with vectors $\pm(\rho_i-\rho_{i-1}+ka)$.
Proposition~\ref{om del rho} says that $a$ is a linear combination of fundamental weights with coefficients $0,\pm1$, so $k\in\integers$.
We will prove that $k=0$, so that $q'_\out=q_\out$ and $q'_\inn=q_\inn$.

Lemma~\ref{closure and reflections} says that every reflection $t\in T$ permutes the set of translation vectors of $(F')^{\pm1}$.
Choose $i$ outer such that $i+1$ is inner, so that, by Proposition~\ref{om del rho}, the $\rho_i$-coordinate of $a$ is $1$.
Then
\begin{equation}\label{reflected trans}
s_i(\rho_i-\rho_{i-1}+ka)=\rho_{i+1}-\rho_i+k(a+\rho_{i+1}-2\rho_i+\rho_{i-1})
\end{equation}
equals $\pm(\rho_j-\rho_{j-1}+ka)$ for some $j$.

If \eqref{reflected trans} equals $-(\rho_j-\rho_{j-1}+ka)$ for some $j$, then 
\[2ka=(-1-k)\rho_{i+1}+(2k+1)\rho_i-k\rho_{i-1}-\rho_j+\rho_{j-1}.\]
Since the $\rho_i$-coordinate of $2ka$ is $2k$, we must have $j=i$, and thus
\[2ka=(-1-k)\rho_{i+1}+2k\rho_i+(-k+1)\rho_{i-1}.\]
But every nonzero coordinate of $2ka$ is $\pm2k$, so $-1-k=\pm2k$ and $-k+1=\pm2k$.
There are no integer solutions to these equations, so
%
\eqref{reflected trans} equals $\rho_j-\rho_{j-1}+ka$ for some $j$. 
Thus
\[(k+1)\rho_{i+1}+(-2k-1)\rho_i+k\rho_{i-1}=\rho_j-\rho_{j-1}.\]
Looking at the term $(k+1)\rho_{i+1}$, we see that either $k=0$ or $k=-2$.
But $k=-2$ is impossible because it gives a term $3\rho_i$ on the left.
We see that $k=0$.
\end{proof}

We now combine results to prove our main result, Theorem~\ref{interval group T L}.
Specifically, Theorem~\ref{F L}.\ref{F L intervals} gives a particular choice of $q_\out$ and $q_\inn$ such that the factored translations $F^{\pm1}$ coincide with the loops~$L$, so that $[1,c]_{T\cup F}$ coincides with $[1,c]_{T\cup L}$.
Theorem~\ref{interval group T L} follows by Theorem~\ref{same interval group}.

\subsection{(Non)Factored translations in affine type $\afftype{C}$}\label{McSul C sec}
In affine type $\afftype{C}$, the horizontal root system is irreducible, so $[1,c]_T$ is a lattice.
(The lattice property in this case has been known since~\cite{Digne2}.)
Accordingly, the constructions of McCammond and Sulway do not factor the translations in $[1,c]_T$.
We now use the model of noncrossing partitions of the two-orbifold disk to characterize translations in $[1,c]_T$ and explain how the combinatorial model also suggests that we should \emph{not} to factor these translations.

\begin{prop}\label{C trans}
Let $c$ be a Coxeter element of the Coxeter group $\Stildes$ of type \textup{$\afftype{C}_{n-1}$}, represented as a signing of $\set{1,\ldots,n-1}$.
Then the translations in $[1,c]_T$ are the permutations $(\!(\cdots\,i\,\,\,i+2n\,\cdots)\!)$ such that $i$ is in the signing.
The corresponding noncrossing partition of the two-orbifold disk has exactly one nontrivial block, a disk containing $i$ and no other numbered point, but containing both orbifold points.
\end{prop}
\begin{proof}
The translation subgroup of $W$ is is generated by elements of the form $\tau_{ij}=(\!(i\,\,\,j)\!)_{2n}(\!(i\,\,\,j+2n)\!)=(\!(\cdots\,i\,\,\,i+2n\,\cdots)\!)(\!(\cdots\,j\,\,\,j-2n\,\cdots)\!)$ for $i\not\equiv j\mod{2n}$ and $\nu_i=(i\,\,-i)_{2n}(i\,\,-i+2n)_{2n}=(\!(\cdots\,i\,\,\,i+2n\,\cdots)\!)$.
One checks that a composition of such elements is $\perm^C(\P)$ for some $\P\in\tNCCc$ if and only if it is $\nu_i$ with $i$ in the signing, and that $\P$ is the noncrossing partition described in the proposition.
\end{proof}

The cycle notation for translations in $\Stildes$ suggests no reasonable factorization.
One could factor $(\!(\cdots\,i\,\,\,i+2n\,\cdots)\!)$ into two permutations $(\cdots\,i\,\,\,i+2n\,\cdots)$ and $(\cdots\,-i\,\,\,-i-2n\,\cdots)$.
To see why such a factorization is unreasonable, one can construct a root system for $\Stildes$ and the vector space $V$ spanned by the roots.
(A construction of the root system and $V$ is found, for example, in~\cite[Section~2.2]{affncD}.)
One can then show that the factors $(\cdots\,i\,\,\,i+2n\,\cdots)$ and $(\cdots\,-i\,\,\,-i-2n\,\cdots)$ cannot be realized as transformations of $V$.
By contrast, in type $\afftype{A}$, factored translations act on the space~$V$ spanned by the root system.

\renewcommand{\afftype}[1]{{\widetilde{\raisebox{0pt}[5pt][0pt]{#1}}}}

\end{document}